\documentclass[11pt,reqno,a4paper]{amsart}
\usepackage[british]{babel}

\usepackage{graphicx}
\usepackage{amscd}
\usepackage{amsmath}
\usepackage{amsfonts}
\usepackage{amssymb}
\usepackage{comment}
\usepackage{color}
\usepackage[usenames,dvipsnames,table,xcdraw]{xcolor}
\usepackage{pdfsync}
\usepackage{bbm}
\usepackage{dsfont}
\usepackage[colorlinks]{hyperref}
\usepackage{enumerate}
\usepackage{booktabs}
\usepackage{float}
\usepackage{caption}
\usepackage{subcaption}
\usepackage{longtable}
\usepackage{listings}
\usepackage[T1]{fontenc}
\usepackage[scaled]{beramono}

\definecolor{trp}{rgb}{1,1,1}

\definecolor{red}{rgb}{1,0,.2}

\newtheorem{theorem}{Theorem}[section]
\theoremstyle{plain}

\newtheorem{claim}{Claim}

\newtheorem{conjecture}{Conjecture}

\newtheorem{lemma}[theorem]{Lemma}

\newtheorem{prop}[theorem]{Proposition}
\newtheorem{remark}[theorem]{Remark}

\numberwithin{equation}{section}

\newcommand{\ii}{\mathbf{i}}
\newcommand{\jj}{\mathbf{j}}

\begin{document}
\title[The Prime Grid]{The Prime Grid. Introducing a geometric representation of natural numbers}

\author{Istv\'an Kolossv\'ary}
\address{Istv\'an Kolossv\'ary, Boston University, Department of Biomedical Engineering, \newline Boston, MA 02215, U. S. A.;  \newline Budapest University of Technology and Economics, Department of Chemistry} \email{ikolossv@bu.edu}

\author{Istv\'an Kolossv\'ary}
\address{Istv\'an Kolossv\'ary, Budapest University of Technology and Economics,  Institute of Mathematics, Department of Stochastics;  \newline MTA Alfr\'ed R\'enyi Institute of Mathematics} \email{istvanko@math.bme.hu}

\thanks{ 2010 {\em Mathematics Subject Classification.} Primary 11N64 11K45 Secondary 11N56 11Y70 11-04 37B10 90C56
\\ \indent
{\em Key words and phrases.} prime grid, prime gaps, distribution of primes, prime number theorem, Polignac's conjecture, Markov shift, differential evolution}

\begin{abstract}
In this report we present an off-the-number-line representation of the positive integers by expressing each integer by its unique prime signature as a grid point of an infinite dimensional space indexed by the prime numbers, which we term the prime grid. In this space we consider a zigzag line, termed the number trail that starts at the origin (representing 1) and travels through every single grid point in the order of the increasing sequence of the natural numbers.

Using the infinity norm we define an arithmetic function $L_\infty (N)$ tabulating the total length of the zigzag up to the integer $N$. We show that $L_\infty(N)$ grows linearly in $N$. Based on computing $L_\infty$ up to $N=10^{12}$ we conjecture the exact rate of growth, which we substantiate analytically by constructing a series of Markov shifts that give gradually better approximations.

Our other interest is looking at the prime gaps along $L_\infty$, i.e. the differences $L_\infty (p_{i+1} )-L_\infty (p_i )$ between every two consecutive primes. After some preliminary observations we extend this analysis to second order differences (difference of differences) as well. Based on our computations up to $N=10^{12}$ we conclude that the distribution of prime numbers along the number trail shows a considerably richer structure compared to their distribution on the traditional number line. We also formulate modified versions of the prime number theorem and Polignac's conjecture.
\end{abstract}
\date{\today}

\maketitle

\thispagestyle{empty}

\section{Introducing the prime grid}

Natural numbers represent our concept of counting things and they are deeply rooted in the conscious and unconscious. Natural numbers, naturally feel familiar, yet under closer scrutiny they present some of the most difficult mathematical problems that have engaged the greatest minds over millennia. Without a doubt prime numbers and their distribution among the natural numbers have puzzled and fascinated most experts and non-experts.

Traditionally natural numbers are plotted as evenly spaced points on the number line. This representation gives a clear idea of their magnitude but not much else. Another possible method of visualizing the natural numbers is called the Ulam spiral \cite{UlamSpiral}, where the numbers are placed in increasing order on the grid points of $\mathbb{Z}^2$ starting with 1 at $(0,0)$ and spiraling outwards, so $2\to(1,0),\, 3\to(1,1),\, 4\to(0,1),\, 5\to(-1,1),\, 6\to(-1,0)$ and so on. Interestingly, prime numbers tend to line up along certain diagonal lines corresponding to specific quadratic polynomials, which could be explained by a conjecture of Hardy and Littlewood \cite{Hardy1923} (if proven to be true).

In this report we present a different, off-the-number-line spatial representation of positive integers whereby the numbers are laid out on a multidimensional grid--assembled from powers of the prime numbers. In number theory, the fundamental theorem of arithmetic states that every natural number $N>1$ is either a prime number itself or it is the product of a unique set of prime numbers, each raised to some fixed power $\geq1$. Because the factorization formula is unique, it suggests a coordinate system of sorts, to represent $N$ in a space ``spanned'' by prime numbers.

\subsection{The prime grid}
The factorization theorem states that every natural number $N$ can be uniquely identified with an infinite sequence $\ii^{N}=\ii=(i_1,i_2,\ldots)$ of non-negative integers, called the prime signature of $N$, so that
\begin{equation*}
N= p_1^{i_1}p_2^{i_2}\ldots p_k^{i_k}\ldots \, ,
\end{equation*}
where $\{p_k\}_{k=1}^\infty$ denote the prime numbers in ascending order. The number $1$ is represented by the sequence $\mathbf{0}=(0,0,\ldots)$. This gives a bijection between the positive integers and the space of all infinite sequences that consist only of zeros except for a finite number of positive integers. We can think of $\mathbf{0}$ as the origin of a coordinate system whose axis are indexed by the prime numbers and each natural number is a grid point on this \textit{prime grid}. It is a vector space, where the addition of two signatures $\ii+\jj$ (corresponding to $N$ and $M$) and also the multiplication by a scalar $n\in\mathbb{N}$ is done coordinate-wise. The former represents multiplication $NM$, while the latter raises $N$ to the $n$th power. In further notation we will freely interchange $N$ with its signature $\ii$.

\begin{remark}
Notice that we could canonically extend our grid to represent all positive rational numbers by allowing negative integers in a signature as well. In this case subtraction of two signatures could also be defined coordinate-wise and would represent division among the natural numbers. We consider only non-negative signatures henceforth.
\end{remark}

In this vector space we can consider the $\ell_1$ and the $\ell_\infty$ norms, i.e.
\begin{equation*}
\|N\|_1=\|\ii\|_1=|i_1|+|i_2|+\ldots \;\text{ and }\; \|N\|_\infty=\|\ii\|_\infty=\max \{|i_1|,|i_2|,\ldots\}.
\end{equation*}
In other words, $\|N\|_1$ counts the total number of prime factors (with multiplicity) of $N$. This is usually referred to as the arithmetic function $\Omega(N)$ in number theory. Further relations to other arithmetic functions include
\begin{equation}\label{eq:N1=Omega}
\omega(N)\leq \Omega(N)=\|N\|_1 \leq \sigma_0(N)-1,
\end{equation}
where $\omega(N)$ counts the number of distinct prime factors of $N$ and $\sigma_0(N)=\prod(i_k+1)$ is the divisor function counting the number of divisors of $N$ ($\sigma_0(N)$ also counts 1 as a divisor, which does not appear in the signature of $N$).

With either of these norms we can define a metric in the usual way. For instance, the $\ell_\infty$ metric for two natural numbers $N$ and $M$ with signatures $\ii$ and $\jj$, also referred to as their Chebyshev distance is equal to
\begin{equation*}
d_{\infty}(N,M) = d_{\infty}(\ii,\jj) = \|\ii-\jj\|_\infty= \max\{|i_1-j_1|,|i_2-j_2|,\ldots\}.
\end{equation*}
 For example, the numbers at unit distance from the origin in the $\ell_1$ metric are the prime numbers, while in the $\ell_\infty$ metric they include exclusively all the square-free numbers. In general, the $k$th power-free numbers are enclosed in the sphere of radius $k-1$ (centered at $\mathbf 0$) using the $\ell_\infty$ metric. We refer to the numbers $M$ for which $\|M\|_\infty=k$ as the Chebyshev contour at distance $k$. Furthermore, it is well-known that the number of $k$-free numbers up to $M$ follow an asymptotic of $M/\zeta(k)+ O(\sqrt[k]{M})$, where $\zeta(k)$ is the Riemann zeta function and $O$ is the usual big-$O$ notation. Hence, the Chebyshev contours follow the asymptotic
\begin{equation}\label{eq:vsz_Ninfty=k}
\lim_{N\to\infty}\frac{|\{M\leq N \text{ such that } \|M\|_\infty =k \}|}{N} = \frac{1}{\zeta(k+1)} - \frac{1}{\zeta(k)}.
\end{equation}
For large $N$ we interpret this as the probability that a uniformly chosen $M\leq N$ has $\ell_\infty$ norm equal to $k$, see further below \eqref{eq:def_q_k}.
 
\subsection{The number trail: an analog of the number line}

We focus on a particular object on the prime grid, a zigzag path that starts at $\mathbf 0$ and crisscrosses through every single grid point on the prime grid in the order of the increasing sequence of the natural numbers. We term this zigzag path the \textit{number trail} and define an arithmetic function $L(N)$ tabulating the total length of the number trail up to $N$ using the chosen metric. That is
\begin{align}
L_\infty (N) &=\sum_{K=1}^{N-1}d_\infty(K\!+\!1,K) = \sum_{K=1}^{N-1} \max\{\|K\!+\!1\|_\infty,\|K\|_\infty\},  \label{def:Linfty} \\
L_1 (N) &= \sum_{K=1}^{N-1}d_1(K\!+\!1,K) = \sum_{K=1}^{N-1} \|K\!+\!1\|_1 + \|K\|_1 = \|N\|_1+2\sum_{K=2}^{N-1}\|K\|_1, \label{def:L1}
\end{align}
where the second equality in each line holds, since $K+1$ and $K$ are always coprime. $L(1)=0$ by definition. We think of it as an analog of the traditional number line, where $L(N)=\sum_{K=0}^{N-1}|(K+1)-K|=N$. It is obvious that $1\leq \|N\|_\infty\leq \|N\|_1$ for every $N$, hence
\begin{equation*}
N < L_\infty(N) < L_1(n) \;\text{ for every } N\geq 4.
\end{equation*}
Asymptotically though $L_\infty$ and $L_1$ behave differently.

\begin{conjecture}\label{conj:Linfty/Nexists}
The limit
\begin{equation}\label{eq:ConjectureLinfty/N}
\lim_{N\to\infty}\frac{L_\infty(N)}{N} =: c_0 
\end{equation}
exists and based on our computations $c_0\approx 2.2883695\ldots .$
\end{conjecture} 
We provide both numerical and analytical evidence that corroborate the conjecture in Section \ref{sec:Linfty}. It is not difficult to show that there exists constants $c$ and $C$ such that $c N<L_\infty(N)<C N$. Furthermore, we construct a series of probabilistic models, namely Markov shifts, with which we can give progressively better lower bounds for $\liminf L_\infty(N)/N$. Appendix \ref{sec:Ccode} contains the programming code used to generate the random sequences. We also formulate a modified version of the prime number theorem, which is supported by our computations.

\begin{prop}\label{prop:BoundsOnL1}
Let $\gamma$ denote the Euler--Mascheroni constant. Then
\begin{equation}\label{eq:BoundL1}
2N\Big(\log\log N - \log\frac{\pi^2}{6}\Big) \leq\; L_1(N)+\|N\|_1 \;\leq 2N\log N - 4(1-\gamma)N+ O(\sqrt N).
\end{equation}
\end{prop}
\begin{proof}
From \eqref{eq:N1=Omega} and \eqref{def:L1} it follows that
\begin{equation*}
2\sum_{K=2}^N \omega(K) \leq L_1(N)+\|N\|_1\leq 2\Big( \sum_{K=2}^{N} \sigma_0(K)-(N-1)\Big).
\end{equation*}
Dirichlet showed that the leading behavior of the divisor summatory function $\sum \sigma_0(K)$ is $N\log N +(2\gamma-1)N+O(\sqrt N)$. This gives the upper bound in \eqref{eq:BoundL1}. Let $\chi(\cdot)$ denote the indicator function. As for the lower bound
$$
\sum_{K=2}^N \omega(K) = \sum_{p,q} \chi(pq\leq N \text{ and } p \text{ prime}) = \sum_{p\leq N} \sum_{q\leq N/p} 1 = \sum_{p\leq N} \frac{N}{p}.
$$
The sum of the reciprocals of the prime numbers can be bounded from below by $\log\log N - \log \pi^2/6$, which concludes the proof.
\end{proof}

\begin{remark}
From now on we will only work with $L_\infty$. There are several reasons why we think using the  $\ell_\infty$ metric is the natural choice.
\begin{itemize}
\item In $\ell_1$ only the primes are separated naturally. The spheres in $\ell_\infty$ however, correspond to the $k$-free numbers. This allows to interpret \eqref{eq:vsz_Ninfty=k} as the probability that for a uniformly chosen $1\leq M\leq N$ ($N$ large) $\|M\|_\infty=k$. This is key in constructing our probabilistic models in Section \ref{sec:Linfty}. 
\item It immediately follows from \eqref{def:L1} that $L_1(p)$ is odd for any prime $p$. This is not the case with $L_\infty$, see Table \ref{table:SmallValuesOfDs}.
\item Conjecture \ref{conj:Linfty/Nexists} and Proposition \ref{prop:BoundsOnL1} show that $L_1$ grows faster than $L_\infty$, which limits the extent of its computability.
\end{itemize}
Henceforth we will usually suppress $\infty$ from the subscript in our notation.
\end{remark}

\subsection{Prime gaps on the number trail}
Our original motivation came from the distribution of prime numbers, for which the most widely used representation uses the prime gap functions
\begin{equation*}
D_k^1=p_{k+1}-p_k \;\text{ and }\; D_k^2=D_{k+1}^1-D_k^1.
\end{equation*}
See Figure \ref{fig:Histogram_D1_D2} in Section \ref{sec:D1andD2} for a histogram of $D^1$ and $D^2$. $D^1$ has been extensively studied, however there are many unanswered conjectures as well. Here we highlight just some of the most recent results. Roughly speaking, the prime number theorem states that on average $D^1_k$ scales as $\log p_k$. However, the fluctuations from the average can be great. Namely, there are subsequences along which the ratio $D^1_k/\log p_k$ can get arbitrarily large \cite{West31} and close to zero \cite{PintzI} as well, i.e.
\begin{equation*}
\limsup_{k\to\infty} \frac{D^1_k}{\log p_k} = \infty \;\text{ and }\; \liminf_{k\to\infty} \frac{D^1_k}{\log p_k} = 0.
\end{equation*}
The latter, for small gaps between primes was improved by the same authors, Goldston--Pintz--Y{\i}ld{\i}r{\i}m in \cite{PintzII2010}. In fact, Banks--Freiberg--Maynard \cite{Banks2016} showed that at least $12.5\%$ of all non-negative real numbers were limit points of $D^1_k/\log p_k$. Further in this direction, a big breakthrough was Zhang's work  \cite{Zhang2014BoundedGaps}, who proved that $D^1_k$ was bounded from above by a constant for infinitely many $k$. The original constant of $7\times 10^7$ has been greatly reduced by work of the Polymath Project and Maynard \cite{Maynard2015}. The conjectured value of the constant is 2, which is the twin prime conjecture.

In the other direction, the best result to date for long gaps between primes is due to Ford--Green--Konyagin--Maynard--Tao \cite{Taoetal2014arXiv}. They showed that for sufficiently large $N$
$$ \max_{p_k\leq N} D^1_k \gg \frac{\log N \log\log N \log\log\log\log N}{\log\log\log N}.$$

In this report we propose a different approach to study the distribution of prime numbers by looking at the prime gaps instead along the number trail on the prime grid. Analogous to the prime gap functions $D^1$ and $D^2$, we define the differences $\mathcal D^1$ and $\mathcal D^2$ along the number trail with $L_\infty$ to be
\begin{equation}\label{def:PrimegapOnNumbertrail}
\mathcal D_k^1:= L_\infty(p_{k+1})-L_\infty(p_k)  \;\text{ and }\; \mathcal D_k^2:= \mathcal D_{k+1}^1 -\mathcal D_k^1.
\end{equation}
For reference, we give a list of the values of $D_k^1, D_k^2, \mathcal D_k^1$ and $\mathcal D_k^2$ for the first few values of $k$. Figures \ref{fig:Histogram_DD1} and \ref{fig:Histogram_DD2} in Section \ref{sec:D1andD2} show detailed histograms of $\mathcal D^1$ and $\mathcal D^2$.
\begin{table}[H]
\caption{Values of $D_k^1, D_k^2, \mathcal D_k^1$ and $\mathcal D_k^2$ for the first few values of $k$.} \label{table:SmallValuesOfDs}
\begin{center}
\begin{tabular}{lrc|rrrr||lrc|rrrr@{}}
	\toprule
	$k$ & $p_k$ & $L_\infty(p_k)$ & $D_k^1$ & $D_k^2$ & $\mathcal D_k^1$ & $\mathcal D_k^2$ & $k$ & $p_k$ & $L_\infty(p_k)$ & $D_k^1$ & $D_k^2$ & $\mathcal D_k^1$ & $\mathcal D_k^2$ \\ \midrule
\rowcolor[HTML]{E8E8E8}
	1 & 2 & 1 & 1 & 1 & 1 & 3 &   6 & 13 & 21 & 4 & -2 & 10 & -6 \\
	2 & 3 & 2 & 2 & 0 & 4 & -2 &  7 & 17 & 31 & 2 & 2 & 4 & 2 \\
\rowcolor[HTML]{E8E8E8}
	3 & 5 & 6 & 2 & 2 & 2 & 7 &   8 & 19 & 35 & 4 & 2 & 6 & 10 \\
	4 & 7 & 8 & 4 & -2 & 9 & -5 & 9 & 23 & 41 & 6 & -4 & 16 & -14 \\
\rowcolor[HTML]{E8E8E8}
	5 & 11 & 17 & 2 & 2 & 4 & 6 & 10& 29 & 57 & 2 & 4 & 2 & 14 \\
	\bottomrule
\end{tabular}
\end{center}
\end{table}
The most apparent difference between $D^1, D^2$ and $\mathcal D^1, \mathcal D^2$ is that the former only take even values (except for the gap between 2 and 3), whereas the latter can also take odd values. The first instance where $\mathcal D^1=7$ is for primes 43 and 47. In Section~\ref{sec:D1andD2} we provide further observations regarding $\mathcal D^1, \mathcal D^2$ and based on our computations highlight differences in the intricate structure compared to $D^1, D^2$ including a modified version of Polignac's conjecture.

In Appendix \ref{sec:Hist_data} raw numerical data of the histograms of $\mathcal D^1$ and $\mathcal D^2$ are presented in Table \ref{tab:Hist_DD1_1e2-12} and Table \ref{tab:Hist_DD2_1e2-12}, respectively for $N\leq10^{2}$ to $N\leq10^{12}$. Size limitations allow for showing $\mathcal D^1$ differences up to 80 and $\mathcal D^2$ differences between -60 and 60. Note, however, that we list the Sagemath Python code in Appendix \ref{sec:Code} that can be used to generate the entire data set including both numerical and graphical representations of the histograms. A detailed explanation of the code is also provided.

At this point we wish to emphasize that the main objective of this paper is merely to draw the attention of those interested towards this alternative look at the natural numbers. Most of our conclusions require a leap of faith relying mainly on empirical data and simulations. Nevertheless, we consider the mathematical questions posed by the paper interesting in their own right.  

\section{Asymptotic growth of the number trail $L_\infty(N)$}\label{sec:Linfty}

Recall $L_\infty$ from \eqref{def:Linfty}. We begin by showing that $L_\infty$ grows linearly.

\begin{lemma}\label{lemma:Linftyislinear}
Let $c:= \sum_{k=1}^{\infty}(1-1/\zeta(k))=1.7052\ldots$. Then
\begin{equation*}
c \leq \liminf_{N\to\infty}\frac{L_\infty(N)}{N} \leq \limsup_{N\to\infty}\frac{L_\infty(N)}{N} \leq 2c. 
\end{equation*}	
\end{lemma}
\begin{proof}
It readily follows from \eqref{def:Linfty} that
\begin{equation*}
\sum_{K=2}^N \|K\|_{\infty}\leq\; L_\infty(N) \;\leq 2\sum_{K=2}^N \|K\|_{\infty}.
\end{equation*}
The sum $\sum \|K\|_\infty$ can be well approximated using \eqref{eq:vsz_Ninfty=k}
$$
\sum_{K=2}^N \|K\|_{\infty} = \sum_{k=1}^\infty k\cdot|\{K\leq N:\; \|K\|_\infty=k\}| = \sum_{k=1}^\infty kN\Big(\frac{1}{\zeta(k+1)}-\frac{1}{\zeta(k)}\Big)
$$
with an error term of the order $O(\sqrt{N} \log N )$. Interchanging the order of summation and using that $\lim\limits_{k\to\infty}\zeta(k)=1$, we get
$$
\sum_{k=1}^\infty k\Big(\frac{1}{\zeta(k+1)}-\frac{1}{\zeta(k)}\Big) = \sum_{k=1}^{\infty} \sum_{i=k}^{\infty} \Big(\frac{1}{\zeta(i+1)}-\frac{1}{\zeta(i)}\Big) =  \sum_{k=1}^{\infty} 1-\frac{1}{\zeta(k)},
$$
which is a convergent series. Dividing by $N$ completes the proof.
\end{proof}

A closer inspection shows that $L_\infty(N)$ depends on how the $(1/\zeta(k+1)-1/\zeta(k))N$ numbers with $\ell_\infty$ norm equal to $k$ are distributed among $\|2\|_\infty,\|3\|_\infty,\ldots,\|N\|_\infty$. Without going into details
\begin{itemize}
\item If they are bunched together as much as possible, then the lower bound for the $\liminf$ can be improved to
$$ \sum_{k=1}^{\infty}\Big(k+\frac{1}{2^{k+1}-1}\Big) \Big(\frac{1}{\zeta(k+1)}-\frac{1}{\zeta(k)}\Big) = 1.9476\ldots \,.$$
\item If they are spread apart as evenly as possible, then the upper bound for the $\limsup$ can be improved to 
$$ \frac32 \sum_{k=1}^{\infty}\Big(1-\frac{1}{\zeta(k)}\Big) = 2.5578\ldots\,.$$
\end{itemize}
These constants are still quite far from the conjectured value of $c_0=2.2883\ldots\,$. In order to get a better understanding of the deterministic sequence
$$\boldsymbol \Omega = \|2\|_\infty,\|3\|_\infty,\|4\|_\infty,\ldots$$
we consider it as a pseudorandom sequence, in the sense that it behaves like a random sequence generated by some probabilistic model. The idea is very much the same how Cram\'er modeled the prime numbers \cite{Cramér1936}. The goal is to construct a series of models which generate random infinite sequences with prescribed properties that the deterministic one satisfies. Then a typical realization will in some sense well-approximate the deterministic sequence.

\subsection{The sequence $\|2\|_\infty,\|3\|_\infty,\ldots$}
We consider $\boldsymbol{\Omega}=\|2\|_\infty,\|3\|_\infty,\ldots$ as an infinite sequence of letters from the alphabet $\mathcal{A}=\mathbb N$. Let $\underline{\omega}=\omega_1\ldots\omega_n$ denote a word of length $|\underline{\omega}|=n$ from $\mathcal{A}$. There are two key factors which greatly determine the rate of growth of $L_\infty$.
\begin{enumerate}
\item The relative frequency $q_k$ with which a given value $k\in\mathbb N$ appears in $\boldsymbol{\Omega}$. According to \eqref{eq:vsz_Ninfty=k}
\begin{equation}\label{eq:def_q_k}
q_k=\frac{1}{\zeta(k+1)}-\frac{1}{\zeta(k)}.
\end{equation}
\item There's a set $\mathcal F$ of forbidden words which never appear in $\boldsymbol{\Omega}$. Observe that any subsequence of consecutive characters of length $2^n$ must contain at least one element with $\|\cdot\|_\infty\geq n$ for any $n>1$. This defines 
\begin{equation}\label{eq:def_ForbiddenWords}
\mathcal F = \bigcup_{n=1}^\infty \mathcal{F}_n = \bigcup_{n=1}^\infty \{\underline{\omega}:\; |\underline{\omega}|=2^{n+1},\, \omega_i<n+1 \text{ for } 1\leq i\leq 2^{n+1}\}.
\end{equation} 
\end{enumerate}
The first one captures the density of each letter $k\in\mathbb N$, while the second one describes additional structure present in the sequence $\Omega$. 

We believe the words in \eqref{eq:def_ForbiddenWords} are the only forbidden ones, however we do not have a complete proof. Consider an arbitrary $\underline{\omega}=\omega_1\ldots\omega_n\notin \mathcal F$. Choose any $n$ distinct primes $\mathbf p= (p_1,\ldots,p_n)$ and look at the system of congruences
\begin{equation}\label{eq:CongSystem}
x+i-1 \equiv 0\mod p_i^{w_i}, \;\;i=1,\ldots,n.
\end{equation}
Since the $p_i$ are coprime, the Chinese remainder theorem implies that there exists a unique $x$ between $1$ and $M=\prod_{i=1}^n p_i^{w_i}$ which satisfies \eqref{eq:CongSystem}. Then of course $\|x+i-1\|_\infty\geq \omega_i$. If all are equalities, then we found $\underline{\omega}$ in $\boldsymbol{\Omega}$. If not, then we can try with $x+kM$ for some positive integer $k$, since all such numbers satisfy \eqref{eq:CongSystem}. It is unclear whether such a $k$ always exists. It need not exist for all choices of $\mathbf p$. For example, $\underline \omega=111$ never appears in $\boldsymbol{\Omega}$ with $\mathbf p=(2,3,5)$, but when $\mathbf p=(5,2,7)$ it does with $x=5$. For illustration we give examples of $\underline{\omega}$ in $\boldsymbol{\Omega}$ in Table \ref{table:FindWordInOmega}. 

\begin{table}[H]
	\caption{Selected non-trivial $\underline{\omega}$ with their place of appearance in $\boldsymbol{\Omega}$} \label{table:FindWordInOmega}
	\begin{center}
		\begin{tabular}{l|rrr@{}}
			\toprule
			$\underline \omega$ & $x+kM$ & $k$ & $\mathbf p$ \\ \midrule
			\rowcolor[HTML]{E8E8E8}
			$17,30$ & $27\,699\,975\,238\,617\,792\,512$ & $1$ & $(2,3)$ \\
			$1,15,3,14$ & $18\,890\,469\,353\,465\,057\,219\,498$ & $7$ & $(2,3,5,7)$  \\
			\rowcolor[HTML]{E8E8E8}
			$1, 2, 2, 1, 3, 5, 2, 1$ & $93\,377\,215\,627\,231\,323$ & $16$ & $(3, 2, 5, 7, 11, 13, 17, 19)$ \\
			\bottomrule
		\end{tabular}
	\end{center}
\end{table}

\subsection{Random model of $\|2\|_\infty,\|3\|_\infty,\ldots$} 
Let $X_1,X_2,X_3,\ldots$ be a random sequence generated by an ergodic stationary Markov-chain with $X_i$ taking values in $\mathbb N$. Refer to \cite{DurrettBook2010} for basics on Markov-chains. Assume that the chain's stationary distribution $\mathbb P (X_i=k)=q_k$ satisfies \eqref{eq:def_q_k}. Furthermore, there is a set $\mathcal E\subseteq\mathcal F$ of \textit{eliminated words}, which never appear in any sequence generated by the chain, but any finite $\underline{\omega}\notin\mathcal{E}$ appears in $X_1,X_2,X_3,\ldots$ with probability one. 

Motivated by \eqref{def:Linfty},  consider the variables $Y_i:= \max\{X_{i+1},X_i\}$ which, due to stationarity are identically distributed. Then the sum $\sum Y_i$ can be interpreted as a random version of $L_\infty$. The ergodic theorem implies that the ergodic average
\begin{equation*}
\frac 1 N \sum_{i=1}^N Y_i\to \mathbb E Y, \;\text{ as } N\to\infty
\end{equation*}
almost surely and in $L^1$ (for basic facts in ergodic theory, we refer to \cite{WaltersErgodBook}). Whenever necessary, we indicate the dependence of $X$ and $Y$ on $\mathcal E$ by writing $X(\mathcal E)$ and $Y(\mathcal E)$. 
\begin{prop}\label{prop:moreEliminated}
For a sequence of ergodic Markov-chains whose stationary distribution  satisfies \eqref{eq:def_q_k} and $\mathcal E_1\subsetneq\mathcal E_2\subsetneq\ldots\subsetneq\mathcal E_n\subsetneq\ldots\subseteq\mathcal{F} $ the sequence $\mathbb E Y(\mathcal E_n)$ converges as $n\to\infty$. In particular, if $\cup_n\mathcal E_n=\mathcal F$ then denote the limit by $\mathbb E Y(\mathcal F)$.
\end{prop}
\begin{proof}
For two Markov chains in the sequence for which $\mathcal E_n\subsetneq\mathcal E_m$ we have
\begin{equation*}
\mathbb E Y(\mathcal E_n)<\mathbb E Y(\mathcal E_m).
\end{equation*}
This is straightforward, since with probability one, any $\underline \omega \in \mathcal E_m\setminus \mathcal E_n$ will appear in $X_1(\mathcal{E}_n),X_2(\mathcal{E}_n),X_3(\mathcal{E}_n),\ldots$ and the function $\max(x,y)$ is strictly increasing in both variables. Hence, $\mathbb E Y(\mathcal E_n)$ is a strictly increasing sequence of real numbers.

Furthermore, by mimicking the proof of Lemma \ref{lemma:Linftyislinear} we can obtain a deterministic upper bound for $\lim \mathbb E Y(\mathcal E_n)$, and therefore the limit exists and is finite. Indeed, observe that $\sum Y_i\leq 2 \sum X_i$ and due to \eqref{eq:def_q_k} we have $\mathbb E X_i=c$ from Lemma \ref{lemma:Linftyislinear}.
\end{proof}

If $\boldsymbol{\Omega}$ behaves like a pseudorandom sequence, then the main question is what is the relationship between $\mathbb E Y(\mathcal{F})$ and the conjectured value of $c_0$?

Alternatively, we give another construction in which we use infinitely many Markov-chains to generate a single random sequence, which will almost surely converge to the same $\mathbb E Y(\mathcal{F})$. Observe that in $\boldsymbol{\Omega}$ up to $\|2^{n+1}\|_\infty$ the restricted words from $\mathcal F$ are exactly $\mathcal F_1\cup\ldots\cup\mathcal F_n$. The idea is to use the Markov-chain with $\mathcal E_n=\mathcal F_1\cup\ldots\cup\mathcal F_n$ for elements in the random sequence with index between $2^n+1$ and $2^{n+1}$. 
\begin{prop}
Consider the random sequence
\begin{equation*}
X_1,X_2,\ldots = X_1^{(1)},\ldots,X_{2^2}^{(1)},X_1^{(2)},\ldots, X_{2^2}^{(2)},X_1^{(3)},\ldots,X_{2^3}^{(3)},\ldots, X_1^{(n)},\ldots,X_{2^n}^{(n)},\ldots
\end{equation*}
where the block $X_{2^n+1},\ldots,X_{2^{n+1}}=X_1^{(n)},\ldots,X_{2^n}^{(n)}$ is generated by the ergodic Markov-chain with $\mathcal E_n=\mathcal F_1\cup\ldots\cup\mathcal F_n$ started from the stationary distribution given by \eqref{eq:def_q_k}. As before, $Y_i:= \max\{X_i,X_{i+1}\}$. Then almost surely
\begin{equation}
\lim_{N\to\infty}\frac 1 N \sum_{i=1}^{N} Y_i= \mathbb E Y(\mathcal F).
\end{equation} 
\end{prop}
\begin{proof}
Let $Z^{(j)}=2^{-j}\sum_{i=1}^{2^j}Y_i^{(j)}$. Then $\mathbb E Z^{(j)}=\mathbb E Y(\mathcal E_j)$. Introduce the event $A(j):= |Z^{(j)}-\mathbb E Z^{(j)}|>j^{-1}$. From Chebyshev's inequality it follows that
\begin{equation*}
\begin{aligned}
\mathbb P(A(j)) &\leq \frac{j^2}{2^{2j}}\mathbf{Var} \Big(\sum_{i=1}^{2^j}Y_i^{(j)}\Big) \\ &= \frac{j^2}{2^{2j}} \left( 2^j\mathbf{Var}(Y_1^{(j)})+2\sum_{i=1}^{2^j}\mathbf{Cov}(Y_i^{(j)},Y_{i+1}^{(j)})+\mathbf{Cov}(Y_i^{(j)},Y_{i+2}^{(j)})\right)\\
&= \frac{j^2}{2^{j}}  \left( \mathbf{Var}(Y_1^{(j)})+2(\mathbf{Cov}(Y_1^{(j)},Y_2^{(j)})+\mathbf{Cov}(Y_1^{(j)},Y_{3}^{(j)}))\right),
\end{aligned}
\end{equation*}
where in the first equality we used that $Y_i^{(j)}$ are identically distributed and due to the Markov property $\mathbf{Cov}(Y_n^{(j)},Y_m^{(j)})=0$ if $|n-m|>2$. Moreover, a very rough upper bound for $\mathbf{Var}(Y_1^{(j)})$ and both covariances is $O(j^2)$ (since $X_i^{(j)}\leq j+1$). Hence,
\begin{equation*}
\sum_{j=1}^{\infty} \mathbb P(A(j)) < \infty,
\end{equation*}
thus the Borel--Cantelli lemma implies that with probability one $|Z^{(j)}-\mathbb E Z^{(j)}|>j^{-1}$ only for a finite number of $j$. So we can pick a $c$ such that $|Z^{(j)}-\mathbb E Z^{(j)}|<cj^{-1}$ for all $j$. Finally,
\begin{equation*}
\sum_{j=1}^{\log N} 2^j (\mathbb EZ^{(j)} - cj^{-1})\leq \sum_{i=1}^{N} Y_i = \sum_{j=1}^{\log N} 2^j Z^{(j)} \leq  \sum_{j=1}^{\log N} 2^j (\mathbb EZ^{(j)} + cj^{-1}).
\end{equation*}
Dividing by $N$, the error terms tend to $0$ as $N\to \infty$. The main term is a Cesaro mean with limit $\mathbb E Y(\mathcal F)$, since $\lim\mathbb E Y(\mathcal E_N)=\mathbb E Y(\mathcal F)$ by Proposition \ref{prop:moreEliminated}.
\end{proof}

\subsection{Examples of eliminated words}\label{subsec:examples}

Here we show how to construct the Markov-chain for a few concrete examples of $\mathcal{E}$. We also calculate the corresponding $\mathbb E Y(\mathcal{E})$ and compare it to $c_0$.

\subsubsection*{Eliminated words $\mathcal E=\emptyset$} Then the Markov-chain corresponds to the sequence of independent and identically distributed (i.i.d.) discrete random variables $X_1,X_2,\ldots$  taking the values $1,2,\ldots$ according to the probability vector $\mathbf q= (q_1,q_2,\ldots)$. It's an easy exercise to determine the distribution of $Y_i$ and get that
\begin{equation*}
\frac 1 N \mathbb E \sum_{i=1}^N Y_i = \mathbb E Y_i = \sum_{k=1}^{\infty} k\Big(\frac{1}{\zeta^2(k+1)}-\frac{1}{\zeta^2(k)}\Big) = 2.22101\ldots .
\end{equation*}
This is much closer to the conjectured value of $c_0$ in \eqref{eq:ConjectureLinfty/N} than any of the purely analytic results in this section. It's a lower bound for $c_0$, this shows that the density alone does not determine the asymptotic growth of $L_\infty$, but the forbidden words indeed play a role as well.

\subsubsection*{Eliminated words $\mathcal E=\{'1111'\}$} Define the Markov-chain on the state space $\mathbb N \cup \{'11','111'\} $ with transition matrix
$$P=\begin{array}{@{}r|cccccc}
& '11' & '111' & 1 & 2 & 3 & \ldots\\ \hline
'11' & 0 & p_1 & 0 & p_2 & p_3 & \ldots \\ 
'111' & 0 & 0 & 0 & \frac{p_2}{1-p_1} & \frac{p_3}{1-p_1} & \ldots \\
1 & p_1 & 0 & 0 & p_2 & p_3 & \ldots \\ 
2 & 0 & 0 & p_1 & p_2 & p_3 & \ldots \\ 
3 & 0 & 0 & p_1 & p_2 & p_3 & \ldots \\ 
\vdots & \vdots& \vdots & \vdots& \vdots& \vdots& \ddots
\end{array}$$
Whenever it is in state $'11'$ or $'111'$, only a $1$ appears in the sequence itself. After $'111'$ only $2,3,\ldots$ may come, which implies that $1111$ never appears in any sequence generated by the chain. By choosing $p_1,p_2,\ldots$ so that the stationary vector $\pi=\pi P$ satisfies $\pi(1)+\pi('11')+\pi('111')=q_1$ and $\pi(k)=q_k$ for $k\geq 2$, we also achieve that \eqref{eq:def_q_k} holds. $\pi=\pi P$ yields $\pi('111')=p_1^3(1-p_1)/(1-p_1^4)$ from which
\begin{equation*}
\begin{aligned}
q_1 &= \frac{(p_1+p_1^2+p_1^3)(1-p_1)}{1-p_1^4} \;\Longrightarrow\; p_1=0.704518\ldots \\
q_k &= p_k\Big(1+\frac{p_1}{1-p_1}\pi('111')\Big) \;\Longrightarrow\; p_k=(1-p_1^4)q_k.
\end{aligned}
\end{equation*}
From here we can determine the distribution of $Y_i$
\begin{equation*}
\begin{aligned}
\mathbb P(Y_i=1) &= (\pi(1)+\pi('11'))p_1, \\
\mathbb P(Y_i=k) &= q_k\sum_{\ell=1}^{k}p_\ell + p_k\sum_{\ell=1}^{k-1}q_\ell +\pi('111')p_k\frac{p_1}{1-p_1} \;\;\text{ for } k\geq2
\end{aligned}
\end{equation*}
which numerically gives $\mathbb E Y_i=2.26767\ldots$. As expected from the proof of Proposition \ref{prop:moreEliminated} it is greater than in the i.i.d. case.
\begin{remark}
Notice that it is enough to work with a finite state space. The ratio $q_k/p_k$ is constant for $k\geq 2$, due to the fact that the corresponding columns in $P$ have the same structure. It is convenient to consider all these letters as a single character, say $\ast$, and to have a single row/column in $P$ corresponding to $\ast$. 
\end{remark}
\subsubsection*{Eliminated words $\mathcal E=\{'1111', '11121112'\}$} The finite state space is $\{1,2,\ast\} \cup \{'11','111','1112','11121','111211','1112111'\}$ with transition matrix

$$P=\begin{array}{@{}r|ccccccccc}
1 & 0 & p_1 & 0 & 0 & 0 & 0 & 0 & p_2 & p_\ast \\
'11' & 0 & 0 & p_1 & 0 & 0 & 0 & 0 & p_2 & p_\ast  \\ 
'111' & 0 & 0 & 0 & \frac{p_2}{1-p_1} & 0 & 0 & 0 & 0 & \frac{p_\ast}{1-p_1} \\
'1112' & 0 & 0 & 0 & 0 & p_1 & 0 & 0 & p_2 & p_\ast \\ 
'11121' & 0 & 0 & 0 & 0 & 0 & p_1 & 0 & p_2 & p_\ast \\ 
'111211' & 0 & 0 & 0 & 0 & 0 & 0 &  p_1 & p_2 & p_\ast \\ 
'1112111' & 0 & 0 & 0 & 0 & 0 & 0 & 0 & 0 & \frac{p_\ast}{1-p_1-p_2} \\
2 & p_1 & 0 & 0 & 0 & 0 & 0 & 0 & p_2 & p_\ast \\
\ast & p_1 & 0 & 0 & 0 & 0 & 0 & 0 & p_2 & p_\ast
\end{array}$$
From the solution of $\pi=\pi P$ we get that
\begin{equation*}
\begin{aligned}
\pi('111') &= \frac{(1-p_1)p_1^3}{1-p_1^4+(p_1^3+p_1^4+p_1^5+p_1^6)p_2}, \\
\pi('1112111') &= \frac{p_1^6p_2}{1-p_1^4+(p_1^3+p_1^4+p_1^5+p_1^6)p_2}.
\end{aligned}
\end{equation*}
For \eqref{eq:def_q_k} to hold, we need to set
\begin{equation*}
\begin{aligned}
q_1 &= \pi(1)+\pi('11')+\pi('111')+\pi('11121')+\pi('111211')+\pi('1112111'), \\
q_2 &= \pi(2) + \pi('1112') \;\text{ and }\; q_k=\pi(k) \,\text{ for } k\geq3,
\end{aligned}
\end{equation*}
which determine the numerical values of $p_1, p_2$ and $p_k$
\begin{equation*}
\begin{aligned}
p_1 &= 0.7045175\ldots,\;  p_2=0.179836\ldots \;\text{ and }\\
p_k &= q_k \Big(1+\frac{p_1}{1-p_1}\pi('111')+\frac{p_1+p_2}{1-p_1-p_2}\pi('1112111')\Big)^{-1}.
\end{aligned}
\end{equation*}
The distribution of $Y_i$ is
\begin{equation*}
\begin{aligned}
\mathbb P(Y_i=1) &= (\pi(1)+\pi('11')+\pi('11121')+\pi('111211'))p_1, \\
\mathbb P(Y_i=2) &= (\pi(1)+\pi('11')+\pi('1112')+\pi('11121')+\pi('111211'))p_2 \\
&+ \pi('111')\frac{p_2}{1-p_1} + \pi('1112')p_1 + \pi(2)(p_1+p_2), \\
\mathbb P(Y_i=k) &= q_k\sum_{\ell=1}^{k}p_\ell + p_k\sum_{\ell=1}^{k-1}q_\ell +\pi('111')\frac{p_1p_k}{1-p_1} + \pi('1112111')\frac{(p_1+p_2)p_k}{1-p_1-p_2},
\end{aligned}
\end{equation*}
which numerically gives $\mathbb E Y_i=2.270017\ldots$. Already with just two eliminated words the calculations are quite involved. In order to eliminate substantially more forbidden words we did extensive computer simulations, which we present in Subsection~\ref{subsec:CompSimu}. But first, we formulate the construction more generally in Subsection~\ref{subsec:genConstr}.

\subsection{General construction}\label{subsec:genConstr}
Basically, we constructed a vertex shift of finite type to eliminate the words in $\mathcal E$ (for basic definitions in symbolic dynamics we refer to \cite{LindMarcusSymDyn}) and then picked a compatible Markov shift, whose stationary distribution generate the letters with density that we prescribed beforehand.

More precisely, fix an arbitrary finite set of eliminated words $\mathcal E$. Let $\mathcal A$ denote the set of letters which make up the words in $\mathcal E$. The alphabet will be $\mathcal A^\ast=\mathcal A \cup \{\ast\}$. $\mathcal E$ defines a shift of finite type $\mathcal X$ on the alphabet $\mathcal A^\ast$. Assume the longest word in $\mathcal E$ has length $n$. Let $\mathcal E'$ be the collection of all $n$ length words which contain an element from $\mathcal E$ as a subword. $\mathcal E'$  defines the same $n-1$-step shift as $\mathcal E$ does. Then $\mathcal X$ can be represented by a vertex shift on a graph with vertices  $\mathbf i=(i_1i_2\ldots i_n)\in (\mathcal A^\ast)^n\setminus \mathcal E'$ (the allowed words of length $n$) and adjacency matrix $A=(A_{\mathbf i\mathbf j})$, where
$$A_{\mathbf i\mathbf j} = 1, \text{ iff } i_2\ldots i_n=j_1\ldots j_{n-1},$$
otherwise $A_{\mathbf i\mathbf j} = 0$. If at the $m$th step the shift is in state $\mathbf i=(i_1i_2\ldots i_n)$, then set $X_m=i_n$. Hence, any sequence $X_1,X_2,\ldots$ generated by the vertex shift will not contain any eliminated word from $\mathcal E$.

Next, consider all possible Markov shifts compatible with the vertex shift, i.e. the collection of stochastic matrices $P=(P_{\mathbf i\mathbf j})$ such that $P_{\mathbf i\mathbf j}>0$ iff $A_{\mathbf i\mathbf j} = 1$. Note that all such $P$ are irreducible and aperiodic, hence ergodic. Choose $P$ for which the stationary distribution $\pi=\pi P$ satisfies
\begin{equation}\label{eq:piandq_k}
\Pi(k):=\sum_{\substack{\mathbf i\in (\mathcal A^\ast)^n\setminus \mathcal E' \\ i_n=k}}\pi(\mathbf i) = q_k \;\text{ for } k\in\mathcal A, \;\text{ and }\; \Pi(\ast):=\pi(\ast) = 1-\sum_{k\in\mathcal A} q_k.
\end{equation}
Then, with probability one, the relative frequency of the characters in any sequence generated by the stationary Markov-chain with initial probability vector $\pi$ and transition matrix $P$ will in addition satisfy \eqref{eq:def_q_k}. We are uncertain whether for any $\mathcal E$ such a $P$, whose stationary distribution $\pi$ satisfies \eqref{eq:piandq_k} always exists. We found no references in the literature and think that this by itself could be an interesting question. 

This is an inverse problem of sorts, since $P$ is the unknown and the $q_k$ are given. From \eqref{eq:piandq_k} and $\pi=\pi P$ the $q_k$ can be expressed as functions of the entries of $P$. We arrive at a system of polynomial equations, which we are uncertain whether it's always (numerically) solvable. This leads to questions in algebraic geometry, which is beyond the scope of this paper, but could be interesting in its own right.

\begin{remark}
The construction itself is not specific to the sequence  $\|2\|_\infty,\|3\|_\infty,\ldots$. All we used was \eqref{eq:def_q_k} and \eqref{eq:def_ForbiddenWords}. It can be adapted to an arbitrary infinite sequence $\mathbf x=x_1x_2\ldots$ coming from a (finite) alphabet $\mathcal A$ provided
\begin{enumerate}
\item the relative frequency of each $a\in\mathcal A$ in $\mathbf x$ is given by a probability vector $\mathbf q$. That is 
\begin{equation*}
\lim_{n\to\infty} \frac{|\{i: 1\leq i\leq n \text{ such that } x_i =a \}|}{n} = q_a,
\end{equation*}
\item a finite or countable set of forbidden words are given.
\end{enumerate}
\end{remark}

\subsection{Computer simulations}\label{subsec:CompSimu}

In the examples presented in Subsection \ref{subsec:examples} we constructed the simplest possible Markov shift from a probability vector $\mathbf p =(p_1,p_2,\ldots)$. Namely, the random sequence of characters $X_1,X_2,\ldots$ was generated as follows:
\begin{enumerate}
\item Until an eliminated word from $\mathcal{E}$ appears in the sequence, all $X_i$ are sampled independently from $\mathbf p$.
\item When an eliminated word appears, then the last character, say $X_k$ is re-sampled from $\mathbf p$ until it no longer concludes an eliminated word.
\item Repeat from step (1). 
\end{enumerate}
The task was to determine $\mathbf p$ so that the stationary distribution $\Pi$ of the characters of the alphabet (described in \eqref{eq:piandq_k}) would be equal to the desired $\mathbf q$ from \eqref{eq:def_q_k}. Direct solution of the inverse problem seemed prohibitive, except for the few examples given, but considering it as an optimization problem made it tractable. One particular derivative-free optimization method, differential evolution (DE) \cite{PriceBook2005,StornPrice1997DE} has proven in our hands to be a powerful tool to generate such random models.

We carried out the simulations on the alphabet $\mathcal{A}^\ast=\{1,2,\ldots,25,\ast\}$ with eliminated words $\mathcal{E}=\mathcal F_1\cup \mathcal F_2\cup\ldots\cup\mathcal F_{25}$. Our goal was to find a discrete input distribution $\mathbf p$ such that by randomly sampling this distribution and eliminating forbidden words
\begin{itemize}
\item the output distribution $\Pi$ will be close to $\mathbf q$ in the sense that the root-mean-square (RMS) of $\Pi-\mathbf q$ is small and
\item the model will provide a $c_0$ value close to the conjectured $2.2883...$ value in our series $\boldsymbol{\Omega}$. 
\end{itemize}
Observe that depending on how the re-sampling is done in step (2), two different models can give different values of $c_0$ even if their output distributions are the same. For this reason we worked with two models. In Model~1 the re-sampling was done as described in step (2). That is, if $X_k=\ell$ concludes a forbidden word, then $X_k$ is re-sampled from the characters $\{\ell+1,\ldots,25,\ast\}$ with ratios $(p_{\ell+1},\ldots,p_{25},p_{\ast})$. In Model~2 the re-sampling was done deterministically by choosing $X_k=\ell+1$.   

Starting with reasonable upper and lower bounds to $\mathbf p$ and testing several different flavors of DE, we were able to find input distributions $\mathbf p^{(1)}$ and $\mathbf p^{(2)}$, which gave reasonably accurate output distributions $\Pi^{(1)}$ and $\Pi^{(2)}$ for Model~1 and Model~2, respectively:
\begin{equation*}
\mathrm{RMS}(\Pi^{(1)}-\mathbf q)\!=\!\sqrt{\frac{1}{26}\sum_i (\Pi_i^{(1)}-\mathbf{q}_i)^2}=3.3\times 10^{-4}  \text{ and } \mathrm{RMS}(\Pi^{(2)}-\mathbf q)= 8.5\times 10^{-6}.
\end{equation*}
Table \ref{table:SimulatedDistributions} contains the results of this inverse optimization procedure we utilized. With these distributions we generated a total of 100 random sequences per model, each of length $10^8$ and computed statistics for $c_0$. Table \ref{table:StatOfc0} contains the results. Originally we worked only with Model~1, but were unable to find a distribution with better $\ell_2$ error than the one presented. However, the estimate for $c_0$ is worse than the ones in Subsection \ref{subsec:examples}. For this reason we introduced Model~2. The error achieved was much better and resulted in the closest estimate for the conjectured value of $c_0$ so far. This could suggest that the result is very sensitive on the magnitude of the error, so finding a way to achieve a smaller error for Model~1 is still work in progress. We plan to address this inverse optimization problem in greater detail in a separate forthcoming paper. The C code we used to generate the random sequences can be found in Appendix \ref{sec:Ccode}.

\begin{table}[]
	\centering
	\caption{Result of optimization to find input distributions $\mathbf p^{(1)}$ and $\mathbf p^{(2)}$ which give $\Pi^{(1)}$ and $\Pi^{(2)}$ close to desired $\mathbf q$ for Models 1 and 2.}\label{table:SimulatedDistributions}
	\label{my-label}
	\begin{tabular}{lllll}
		\toprule
		$\mathbf q$          & $\mathbf p^{(1)}$       & $\Pi^{(1)}$  & $\mathbf p^{(2)}$       & $\Pi^{(2)}$  \\ \midrule
		\rowcolor[HTML]{E8E8E8} 
		0.6079271019    & 0.778066482     & 0.60725599 & 0.7791783767    & 0.60792069 \\
		0.2239802707    & 0.1808120958    & 0.22329953 & 0.166204795     & 0.22398693 \\
		\rowcolor[HTML]{E8E8E8} 
		0.09203103034   & 0.0344511792    & 0.09135697 & 0.03914430883   & 0.09202952 \\
		0.04044893751   & 0.005501330873  & 0.03988083 & 0.01075578968   & 0.04044403 \\
		\rowcolor[HTML]{E8E8E8} 
		0.01856525184   & 0.0008606020801 & 0.01810934 & 0.003211510904  & 0.01856443 \\
		0.008767263574  & 0.0001788032959 & 0.00860553 & 0.0009983211701 & 0.0087562  \\
		\rowcolor[HTML]{E8E8E8} 
		0.004219345287  & 5.589433871e-05 & 0.00429196 & 0.0003228213903 & 0.00421908 \\
		0.002056431605  & 2.736073937e-05 & 0.00253894 & 0.000100454897  & 0.00205053 \\
		\rowcolor[HTML]{E8E8E8} 
		0.001010780338  & 1.31193597e-05  & 0.00131435 & 3.482384896e-05 & 0.00100937 \\
		0.0004996424286 & 7.594988824e-06 & 0.00076872 & 1.119959966e-05 & 0.00049958 \\
		\rowcolor[HTML]{E8E8E8} 
		0.0002479184928 & 4.286503901e-06 & 0.0004351  & 5.66707383e-06  & 0.00024959 \\
		0.0001233277188 & 3.540683918e-06 & 0.00035849 & 6.236370813e-06 & 0.00012684 \\
		\rowcolor[HTML]{E8E8E8} 
		6.145390691e-05 & 2.973346931e-06 & 0.00030085 & 6.553250804e-06 & 6.628e-05  \\
		3.065708326e-05 & 2.403969944e-06 & 0.0002432  & 6.186574815e-06 & 3.557e-05  \\
		\rowcolor[HTML]{E8E8E8} 
		1.530527483e-05 & 2.812292935e-06 & 0.00028336 & 5.746875828e-06 & 2.025e-05  \\
		7.644886553e-06 & 3.03521793e-06  & 0.00030529 & 3.470847896e-06 & 1.05e-05   \\
		\rowcolor[HTML]{E8E8E8} 
		3.819860619e-06 & 2.459252943e-06 & 0.00024817 & 1.430745957e-06 & 4.78e-06   \\
		1.909069617e-06 & 2.294506947e-06 & 0.00023149 & 1.018855969e-06 & 2.73e-06   \\
		\rowcolor[HTML]{E8E8E8} 
		9.542479522e-07 & 1.121315974e-06 & 0.00011425 & 7.578192773e-07 & 1.78e-06   \\
		4.770283636e-07 & 1.70937596e-07  & 1.707e-05  & 1.379581959e-07 & 5.3e-07    \\
		\rowcolor[HTML]{E8E8E8} 
		2.384823142e-07 & 2.59843894e-07  & 2.492e-05  & 2.503132925e-07 & 4.6e-07    \\
		1.192305334e-07 & 9.695063776e-08 & 9.22e-06   & 5.449547837e-08 & 1.6e-07    \\
		\rowcolor[HTML]{E8E8E8} 
		5.961172722e-08 & 5.074703883e-08 & 4.52e-06   & 5.148644846e-08 & 1.2e-07    \\
		2.980468194e-08 & 1.579332963e-08 & 1.1e-06    & 2.473640926e-08 & 3e-08      \\
		\rowcolor[HTML]{E8E8E8} 
		1.490194901e-08 & 6.098237859e-09 & 4.2e-07    & 5.160489846e-09 & 2e-08      \\
		1.490155355e-08 & 9.831509773e-09 & 3.9e-07    & 5.419687838e-09 & 0          \\ \bottomrule
	\end{tabular}
\end{table}

\begin{table}[H]
	\caption{Statistics of $c_0$ calculated from random sequences generated by Model~1 and Model~2 using the distributions from Table~\ref{table:SimulatedDistributions}.} \label{table:StatOfc0}
	\begin{center}
		\begin{tabular}{lrr}
			\toprule
			$c_0$ & Model~1 & Model~2  \\ \midrule
			\rowcolor[HTML]{E8E8E8}
			Mean & 2.345592 & 2.280233  \\
			Std.Dev. & 8.51e-05 & 6.82e-05  \\
			\rowcolor[HTML]{E8E8E8}
			Min. & 2.34532 & 2.28006  \\
			Max. & 2.3458 & 2.28038  \\
			\rowcolor[HTML]{E8E8E8}
			Median & 2.34559 & 2.28023  \\
			\bottomrule
		\end{tabular}
	\end{center}
\end{table}

\subsection{Direct computation of $L_\infty$}
During our computations we needed to keep track of the values of $L_\infty$ for all primes. The complete program code with explanations is provided in Appendix \ref{sec:Code}. In Table \ref{table:L_inftyN/N} below we list the value of $L_\infty(p_k)/p_k$ for a few values of $k$. Each entry in the row of $L_\infty(p_k)/p_k$ begins with $2.2883\ldots$, so just the following 6 digits are listed in each case. 
\begin{table}[H]
	\caption{The ratio $L_\infty(N)/N$ at few specific primes. In the row of $L_\infty(p_k)/p_k$ all entries begin with $2.2883\ldots$} \label{table:L_inftyN/N}
	\begin{center}
		\begin{tabular}{l|rrrrr@{}}
			\toprule
			$k$ & $10^5$ & $10^6$ & $10^7$ & $10^8$ & $10^9$ \\ \midrule
			\rowcolor[HTML]{E8E8E8}
			$p_k$ & $1\,299\,709$ & $15\,485\,863$ & $179\,424\,673$ & $2\,038\,074\,743$ & $22\,801\,763\,489$ \\
			$L_\infty(p_k)$ & $2\,974\,210$ & $35\,437\,380$ & $410\,589\,942$ & $4\,663\,867\,856$ & $52\,178\,860\,638$ \\
			\rowcolor[HTML]{E8E8E8}
			$L_\infty(p_k)/p_k$ &  ...660881 & $\ldots 697214$ & $\ldots 694596$ & $\ldots 693897$ & $\ldots 695229$ \\
			\bottomrule
		\end{tabular}
	\end{center}
\end{table} 
The ratios differ by less than $10^{-5}$. These are just a few values, but we calculated the ratios for every $10^5$-th prime between $1$ and $10^{12}$. This gave a set of $376\,079$ data points. The minimal value is $2.28836250$ and the maximal is $2.288371417$ giving a difference of $8.9\times 10^{-6}$. A list-plot of the values reveals even more in Figure \ref{fig:ratio}. 

\begin{figure}[H]
	\centering
	\includegraphics[width=0.95\textwidth]{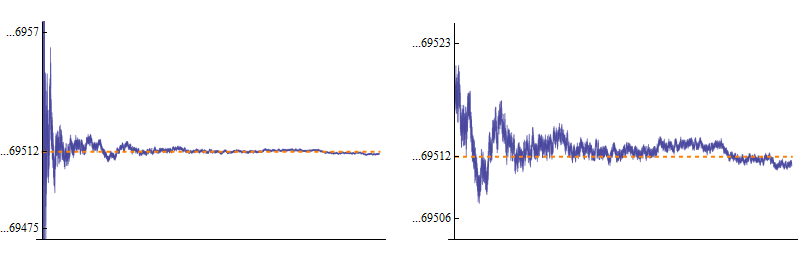}
	\caption{Ratio $L_\infty(p_k)/p_k$ for every $10^5$-th prime between 1 and $10^{12}$. Close-up on right hand side. Each value begins with $2.2883\ldots$}
	\label{fig:ratio}
\end{figure}
There is some oscillation in the beginning, but then very convincingly it seems to settle around a value quite rapidly. Based on these results we conjecture that $c_0\approx 2.28836951\ldots$.
\subsection{Modified prime number theorem}\label{subsec:ModPNT}

Let us modify the prime counting function to count the number of primes up to $N$ on the number trail,
$$\pi_\infty(N):= \max \{k:\; L_\infty(p_k)\leq N\}.$$
Of course, $\pi_\infty(L_\infty(p_k))=k$. If Conjecture \ref{conj:Linfty/Nexists} holds true, then in a way the $\ell_\infty$ norm is just rescaling the traditional number line. Conceivably it is reasonable to think that the asymptotics of $\pi_\infty(N)$ change accordingly.
\begin{conjecture}[Modified Prime Number Theorem]\label{conj:ModPNT}
	\begin{equation*}
	\lim_{N\to\infty}\frac{\pi_\infty(N)}{N/\log N} = \lim_{N\to\infty}\frac{\pi_\infty(N)}{\mathrm{Li}(N)} = \frac{1}{c_0} \approx 0.436992\ldots,
	\end{equation*}
	where $c_0$ is the constant defined by the limit of $L_\infty(N)/N$ in \eqref{eq:ConjectureLinfty/N} and $\mathrm{Li}(x)$ is the offset logarithmic integral function $\int_{2}^{x}1/\ln y \mathrm dy$.
\end{conjecture}
In Table \ref{table:pi_infty(N)} we present the ratios for a few increasing values of $p_k$.
\begin{table}[H]
\caption{Comparing $\pi_\infty(N)$ to $N\log N$ and $\mathrm{Li}(N)$, for $N=L_\infty(p_k)$.} \label{table:pi_infty(N)}
\begin{center}
	\begin{tabular}{c|rcr@{}}
		\toprule
		$k$ & $p_k$ & $\pi_\infty(N)\cdot \log(N) / N$ & $\pi_\infty(N)/\mathrm{Li}(N)$  \\ \midrule
		\rowcolor[HTML]{E8E8E8}
		$10^6$ & $15\,485\,863$ & $0.49053507\ldots$ & $0.46030511\ldots$   \\
		$10^7$ & $179\,424\,673$ & $0.48303924\ldots$ & $0.45721224\ldots$  \\
		\rowcolor[HTML]{E8E8E8}
		$10^8$ & $2\,038\,074\,743$  & $0.47735295\ldots$ & $0.45478434\ldots$  \\
		$10^9$ & $22\,801\,763\,489$ & $0.47294906\ldots$ & $0.45289153\ldots$  \\
		\rowcolor[HTML]{E8E8E8}
		$10^{10}$ &  $252\,097\,800\,623$ & $0.46942719\ldots$ & $0.45136754\ldots$  \\
		$2\cdot10^{10}$ & $518\,649\,879\,439$ & $0.46850132\ldots$ & $0.45096581\ldots$  \\
		\rowcolor[HTML]{E8E8E8}
		$3\cdot10^{10}$ &  $790\,645\,490\,053$ & $0.46798418\ldots$ & $0.45074112\ldots$  \\
		\bottomrule
	\end{tabular}
\end{center}
\end{table} 
As in the case on the traditional number line $\mathrm{Li}(N)$ gives a better approximation to $\pi_\infty(N)$ than $N/\log N$. Both sequences are decreasing nicely, but unfortunately, if there is convergence it seems to be very slow. The numbers do not confirm our conjecture, but they certainly do not contradict it either. We don't know if assuming Conjecture \ref{conj:Linfty/Nexists} an existing proof of the classical prime number theorem could be adapted to this setting to prove Conjecture \ref{conj:ModPNT}.

\section{The prime gap functions $\mathcal D^1$ and $\mathcal D^2$}\label{sec:D1andD2}

We begin the section with simple observations about $\mathcal D_1$, recall \eqref{def:PrimegapOnNumbertrail}, and then point out key differences between $\mathcal D_1$ and $\mathcal D_2$ compared to $D_1$ and $D_2$.

\subsection{Observations about the prime gap function $\mathcal D^1$}

\begin{claim}\label{claim:D1not35}
	$\mathcal D^1$ can take arbitrarily large values. However, $\mathcal D^1$ never takes the values 3 or 5.
\end{claim}
\begin{proof}
	The prime number theorem states that the number of primes up to $N$ grows sub-linearly, hence there are arbitrarily large gaps between two consecutive primes. On the other hand, we know that $ \; L(p_{k+1})-L(p_k)\geq p_{k+1}-p_k.$ This implies the first assertion.
	
	For the second assertion we can assume that $p_k\neq 2$ (since $\mathcal D_1^1=1$). Since $\mathcal D_k^1\geq p_{k+1}-p_k$, $\mathcal D_k^1$ could equal 3 if and only if $p_{k+1}=p_k+2$ and similarly $\mathcal D_k^1$ could equal 5 if and only if $p_{k+1}=p_k+2$ or $p_{k+1}=p_k+4$.
	
	If $p_{k+1}=p_k+2$, then \eqref{def:Linfty} implies
	\begin{equation}\label{eq:D^1twinprime}
	\mathcal D_k^1= \max\{\|p_k\|, \|p_k+1\|\} + \max\{\|p_k+1\|, \|p_k+2\|\} = 2 \|p_k+1\|_\infty,
	\end{equation}
	which is clearly an even number.
	
	If $p_{k+1}=p_k+4$, then either $p_k+1$ or $p_k+3$ is divisible by 4, hence has $\ell_\infty$ norm at least $2$ and so $\mathcal D_k^1\geq 2+2+1+1 >5$.
	
	
\end{proof}
It is also not difficult to characterize $\mathcal D^1$ for the first few small values.
\begin{claim}
	$\mathcal D_k^1$  equals (a) 2, (b) 4, or (c) 6 for some $k$ if and only if
	\begin{enumerate}[(a)]
		\item $p_{k+1}$ and $p_k$ are twin prime and $p_k+1$ is a square-free number,
		\item $p_{k+1}$ and $p_k$ are twin prime and $\|p_k+1\|_\infty=2$,
		\item $p_{k+1}$ and $p_k$ are twin prime and $\|p_k+1\|_\infty=3$ OR \newline $p_{k+1}=p_k+4$ and $\|p_k+2\|_\infty=1$ and either $p_k\!+\!1$ or $p_k\!+\!3$ is divisible by 4.
	\end{enumerate}
	In general, if $p_{k+1}$ and $p_k$ are twin prime and $\|p_k+1\|_\infty=n$ then $\mathcal D_k^1=2n$.
\end{claim}
\begin{proof}
	We saw at the end of the proof of Claim \ref{claim:D1not35} that if  $p_{k+1}-p_k\geq4$ then $\mathcal D_k^1>5$. Hence, if $\mathcal D_k^1=2$ or $4$, then $p_{k+1}$ and $p_k$ must be twin prime. Moreover, \eqref{eq:D^1twinprime} implies that $\|p_k+1\|_\infty=1$ or $2$, respectively. If $\mathcal D_k^1=6$ then either $p_{k+1}$ and $p_k$ are twin prime and $\|p_k+1\|_\infty=3$ or $p_{k+1}=p_k+4$ ($p_{k+1}\geq p_k+6$ not possible). If $p_{k+1}=p_k+4$ then it further follows that one of $p_k+1$ or $p_k+3$ is divisible by 4 and $p_k+2$ is square-free.
	
	The other directions and the last claim are just trivial calculations.
\end{proof}

By definition, $D^1_k=n$ if and only if $p_{k+1}-p_k=n$. However,  $\mathcal D^1$ can take a particular value $n$ in a number of ways. This number is related to the partition function $P(n)$, which counts all the distinct ways $n$ can be written as the sum of natural numbers. Hardy and Ramanujan were the first to show that the asymptotic growth rate of $P(n)$ is $1/(4\sqrt 3n)\exp(\pi\sqrt{2n/3})$. Here there are several restrictions on the partitions. For example, the partition must have an even number of summands (since $p_{k+1}-p_k$ is even) and the largest element of the partition must appear at least twice in it. Also, the partition can not contain any forbidden word from $\mathcal F$ (recall \eqref{eq:def_ForbiddenWords}). A characterization of these partitions or at least some asymptotic result on their number seems difficult. 


\subsection{Comparison of $D^1$ and $D^2$ to $\mathcal{D}^1$ and $\mathcal{D}^2$}

For reference, Fig. \ref{fig:Histogram_D1_D2} shows the familiar histograms of the prime gap functions $D^1$ and $D^2$ up to $N\leq10^{12}$. The semi-regular spiked structure of these histograms has been shown to be attributed to the fact that for every prime number $p$, except 2, $p=\pm 1(\mathrm{mod} 6)$ \cite{Ares2006}.

\begin{figure}[H]
	\centering
	\begin{subfigure}[b]{0.51\textwidth}
		\includegraphics[width=0.95\textwidth]{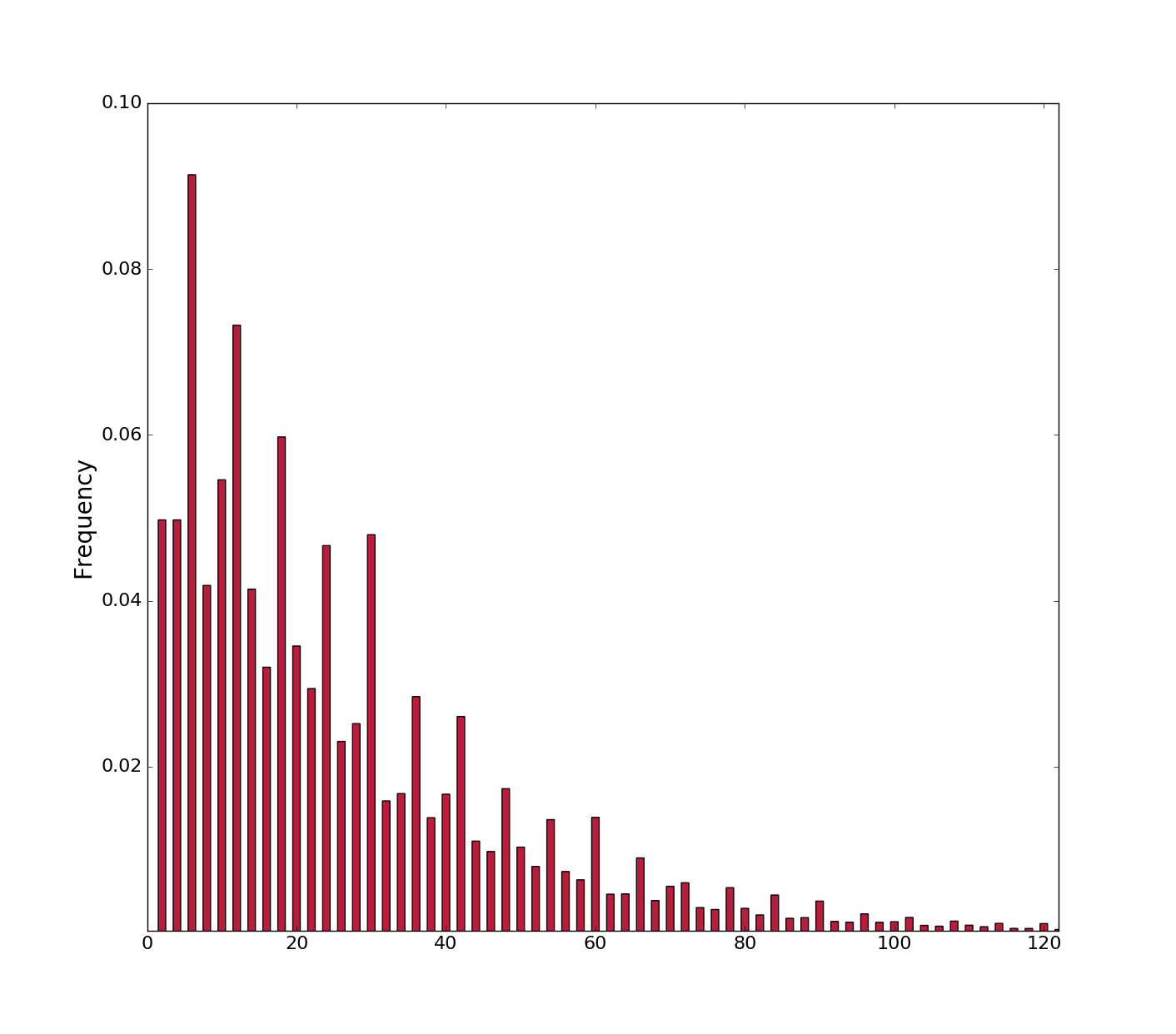}
		\caption{$D_k^1=p_{k+1}-p_k$}
		\label{fig:Hist_D1_1e12}
	\end{subfigure}
	\begin{subfigure}[b]{0.48\textwidth}
		\includegraphics[width=\textwidth]{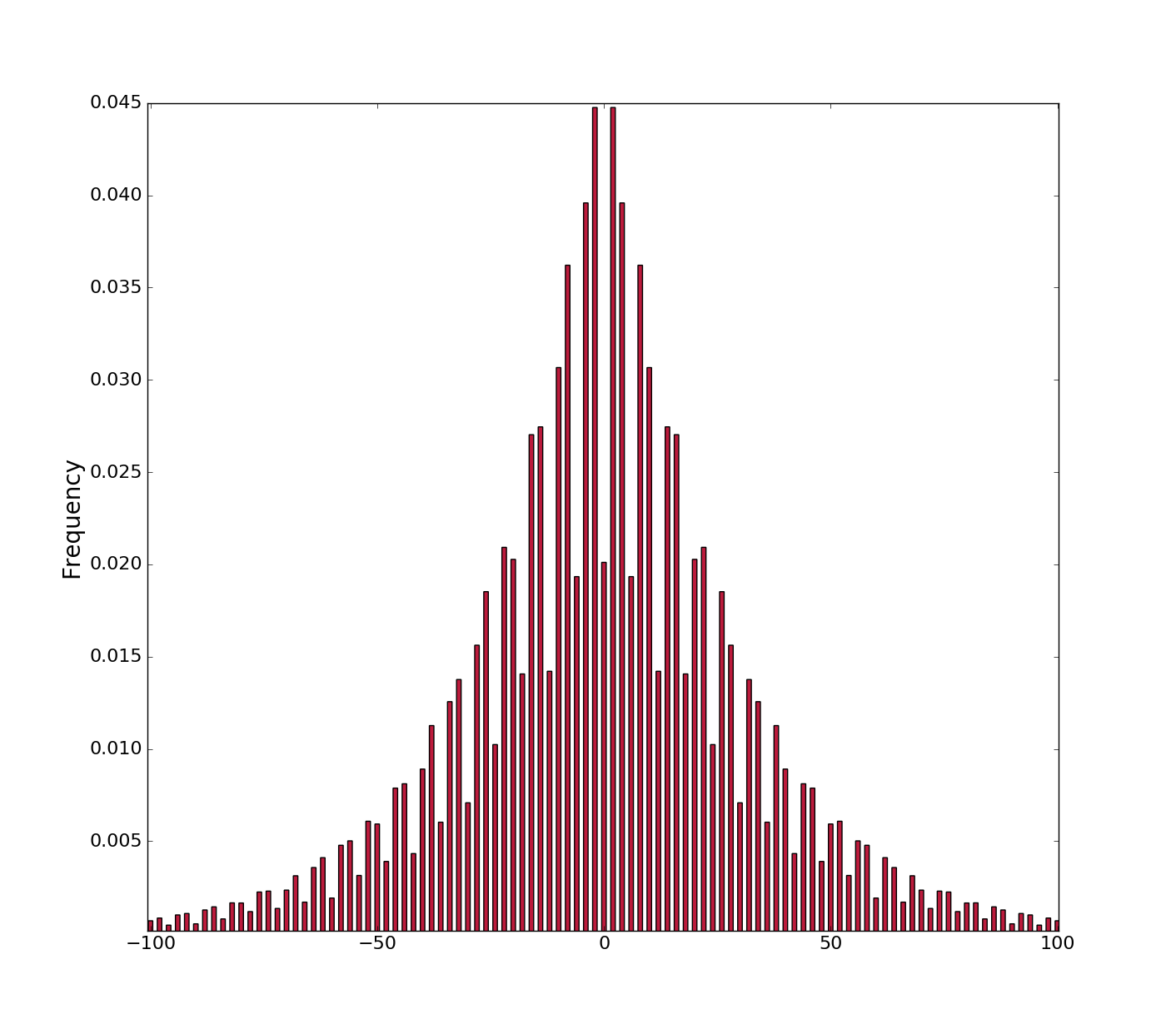}
		\caption{$D_k^2=D_{k+1}^1-D_k^1$}
		\label{fig:Hist_D2_1e12}
	\end{subfigure}
	\caption{Histogram of first \eqref{fig:Hist_D1_1e12} and second \eqref{fig:Hist_D2_1e12} order differences between consecutive prime numbers up to $N\leq 10^{12}$. }\label{fig:Histogram_D1_D2}
\end{figure}
Polignac's conjecture is still open, stating that for any positive even number $N$ there are infinitely many $k$ such that $D^1_k=N$. Huang and Wu \cite{HuangWu2015} recently showed that the set $D=\{d: d$ can be expressed in infinitely many ways as the difference of two primes $\}$ is a $\Delta^*_r$ set for $r\geq 721$, i.e. for any subset $S$ of $\mathbb N$ with $|S|\geq 721$, the intersection of $D$ and the difference set of $S$ is nonempty. The histogram  shows that up to $10^{12}$ the most frequently appearing value of $D^1$ is 6. An integer is called a jumping champion if it is the most frequently appearing value of $D^1$ up to some $N$. The histogram is misleading in the sense that there is a convincing heuristic argument of Odlyzko--Rubinstein--Wolf \cite{ORW99JupingChamps} to support the jumping champion conjecture, stating that the jumping champions are 4 and the primorials $2,6,30,210,\ldots,\prod_{i=1}^{k}p_i,\ldots$. They argue that 30 only becomes jumping champion around $1.74\times10^{35}$. We continue with the histogram of $\mathcal D^1$.

\begin{figure}[H]
	\centering
	\begin{subfigure}[b]{0.51\textwidth}
		\includegraphics[width=0.95\textwidth]{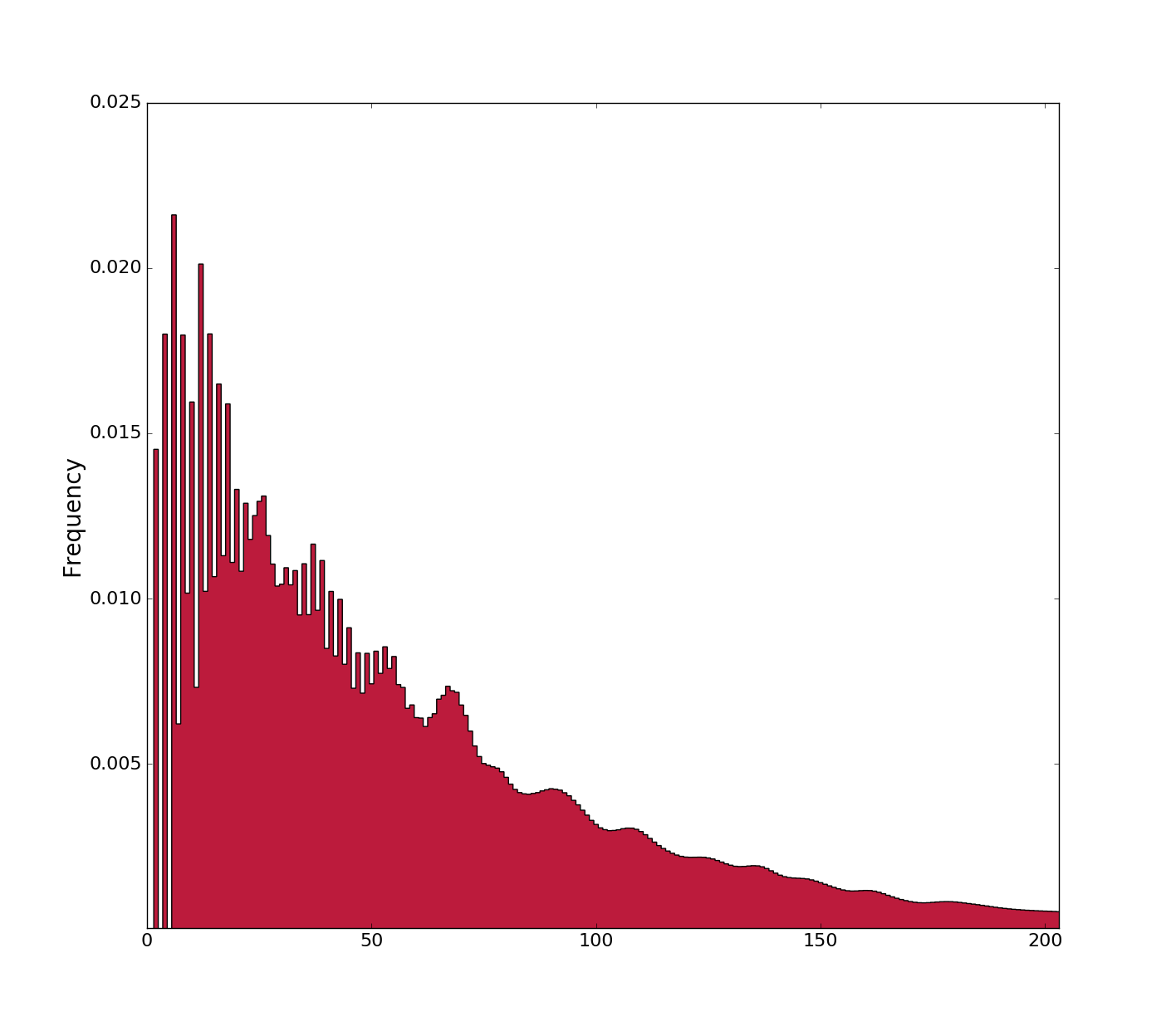}
		\caption{$\mathcal D_k^1=L_\infty(p_{k+1})-L_\infty (p_k)$}
		\label{fig:Hist_DD1_1e12}
	\end{subfigure}
	\begin{subfigure}[b]{0.48\textwidth}
		\includegraphics[width=\textwidth]{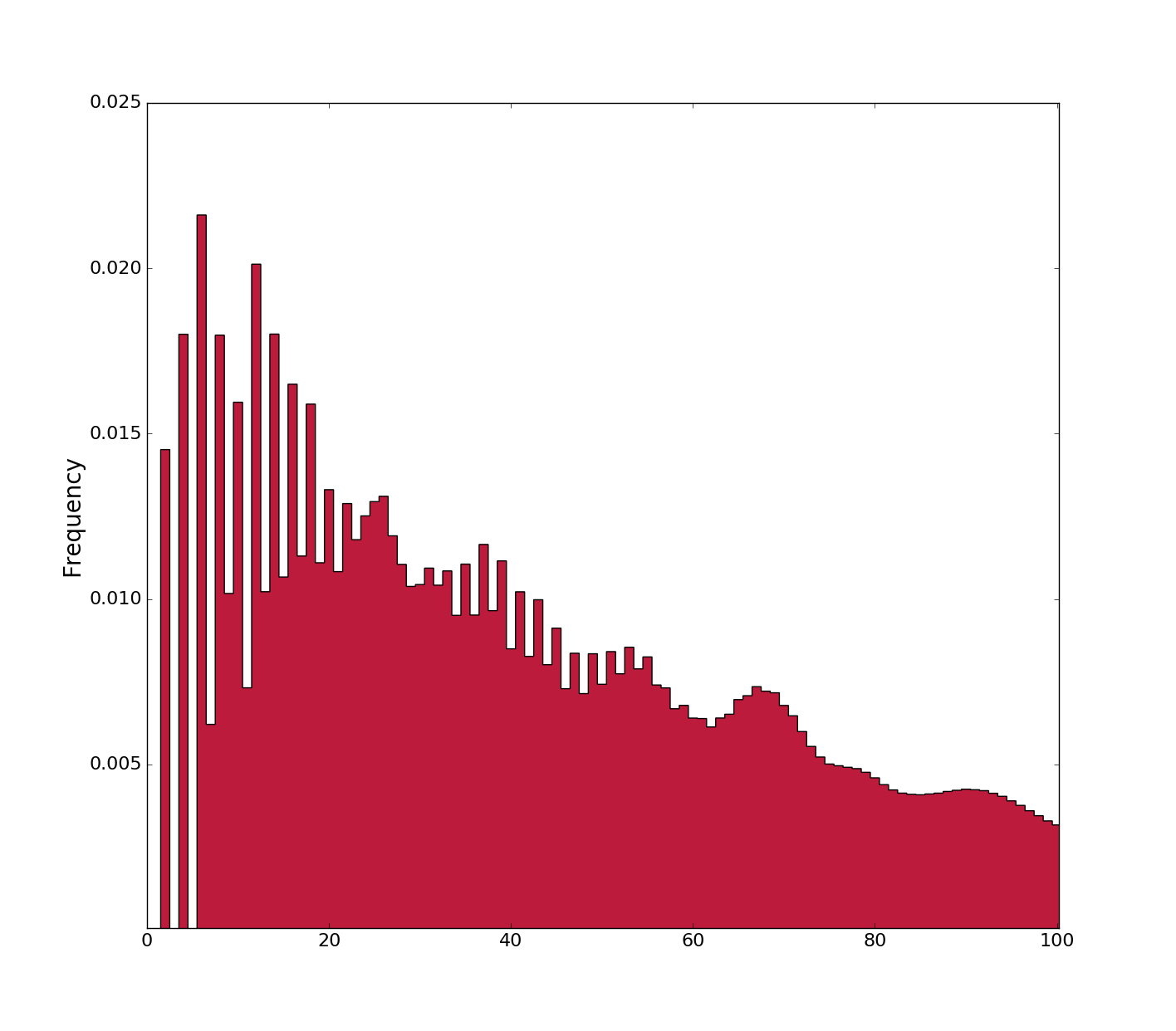}
	\caption{$\mathcal D_k^1$, close-up}
	\label{fig:Hist_DD1_1e12_zoom}
	\end{subfigure}
	\caption{Histogram of first order differences $\mathcal D_k^1$ between consecutive prime numbers on the number trail up to $N\leq 10^{12}$. }
	\label{fig:Histogram_DD1}
\end{figure}
Figure \ref{fig:Histogram_DD1} shows the histogram of $\mathcal D^1$ up to $N\leq 10^{12}$. The jumping champion at this point is likewise 6. As proved in Claim \ref{claim:D1not35}, the values 1,3 and 5 do not appear. However, visibly there are no other excluded values either, so we modify Polignac's conjecture to $\mathcal D^1_k$. 
\begin{conjecture}[Modified Polignac's Conjecture]
	$\mathcal D^1_k$ never takes the values 1,3 or 5, but for any other integer $N>0$, there are infinitely many $k$ such that $\mathcal D^1_k=N$.
\end{conjecture}

The histogram of $\mathcal D^2$ differences is shown in Fig. \ref{fig:Histogram_DD2} up to $N\leq 10^{12}$ with an additional close-up view (\textbf{\ref{fig:Hist_DD2_1e12_zoom}}). The semi-regular spikes seen on $D^1$ and $D^2$ histograms  are replaced by an intricate, non-repeating fine structure on the  $\mathcal D^1$ and $\mathcal D^2$ histograms with numerous, differently shaped local peaks. Most notably, as becomes visible in Fig. \textbf{\ref{fig:Hist_DD2_1e12_zoom}} the cap of the $\mathcal D^2$ histogram is not the highest peak, it has two tiny local maxima at $\pm1$, but there are two more, symmetrical and significantly higher maxima at $\pm5$ and $\pm7$.

\begin{figure}[H]
	\centering
	\begin{subfigure}[b]{0.51\textwidth}
		\includegraphics[width=0.95\textwidth]{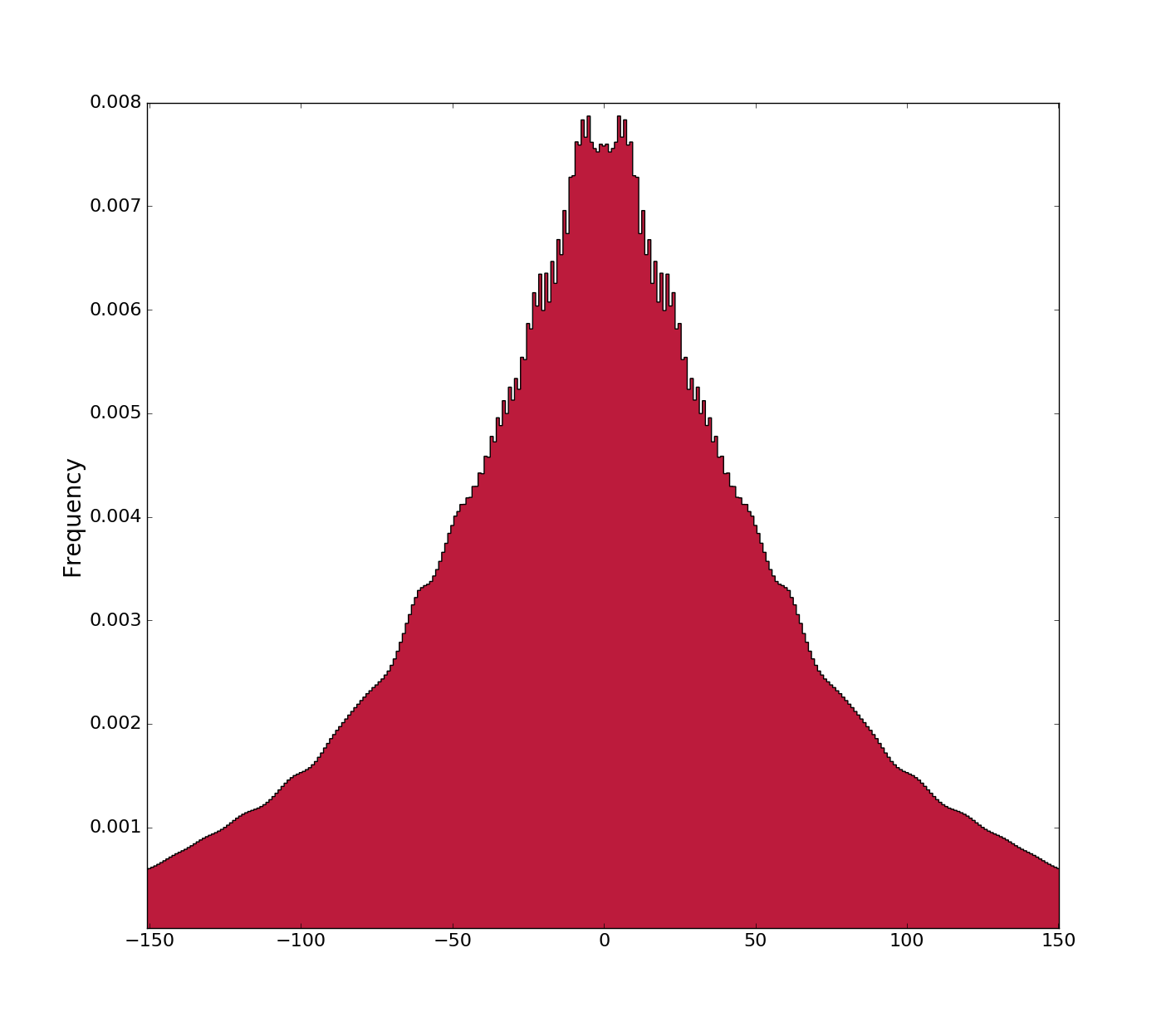}
		\caption{$\mathcal D_k^2=\mathcal D_{k+1}^1-\mathcal D_k^1$}
		\label{fig:Hist_DD2_1e12}
	\end{subfigure}
	\begin{subfigure}[b]{0.48\textwidth}
		\includegraphics[width=\textwidth]{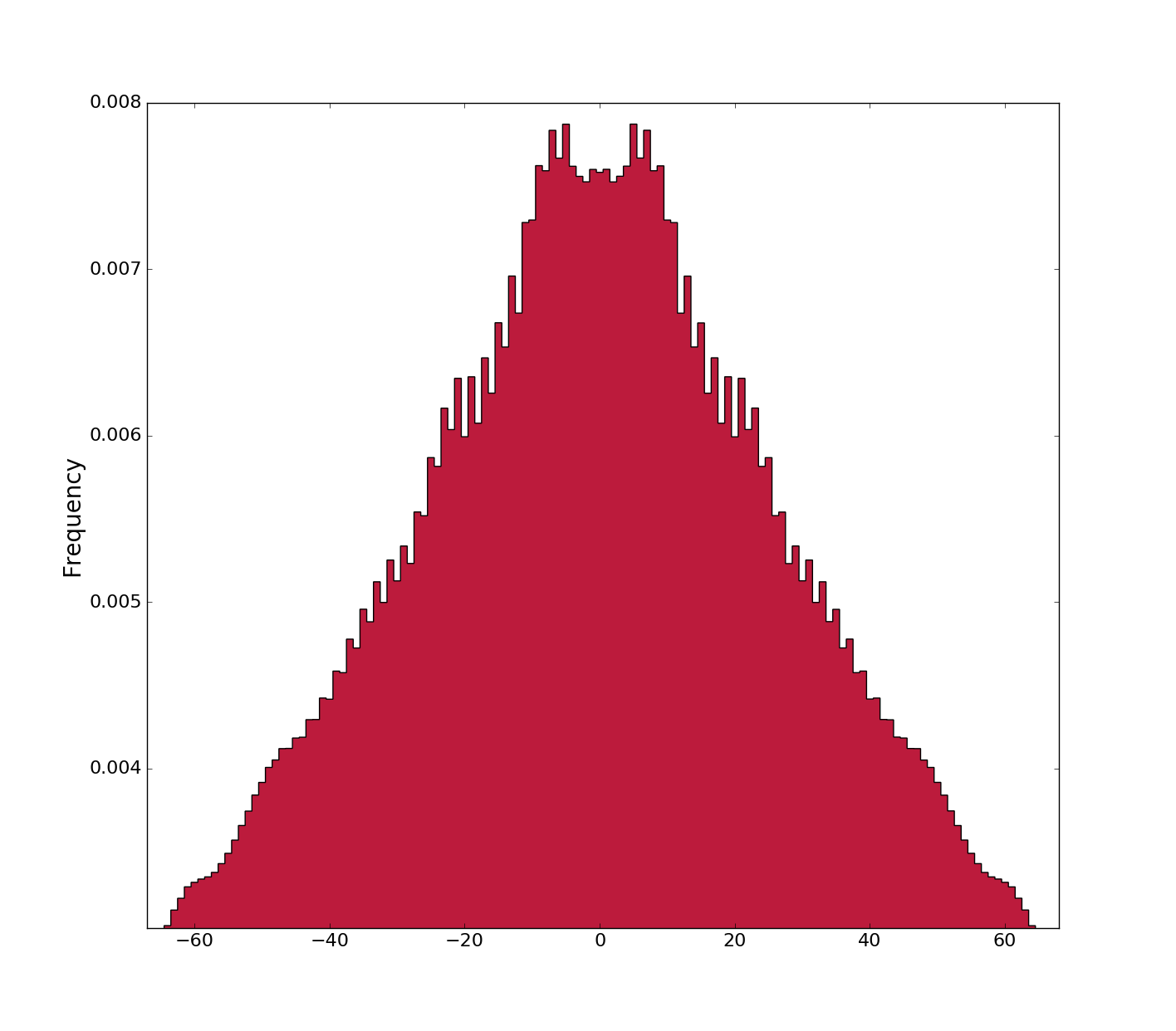}
		\caption{$\mathcal D_k^2$, close-up}
		\label{fig:Hist_DD2_1e12_zoom}
	\end{subfigure}
	\caption{Histogram of second order differences $\mathcal D_k^2$ between consecutive prime numbers on the number trail up to $N\leq 10^{12}$. }
	\label{fig:Histogram_DD2}
\end{figure}

A more detailed view of the fine structure of $\mathcal D^2$ histograms can be displayed on stacked, logarithmic plots shown in Fig. \ref{fig:Stacked_DD2}. 11 histograms, each with a tenfold increase in length $N\leq10^{2-12}$, are stacked from the front toward the back with increasing size. This representation highlights the differences between consecutive histograms and provides good visual evidence to suggest that the histogram of $\mathcal D^2$ decays exponentially at a rate which could converge to some specific number as $N\to\infty$. 

\begin{figure}[H]
	\centering
	\begin{subfigure}[b]{0.51\textwidth}
		\includegraphics[width=0.95\textwidth]{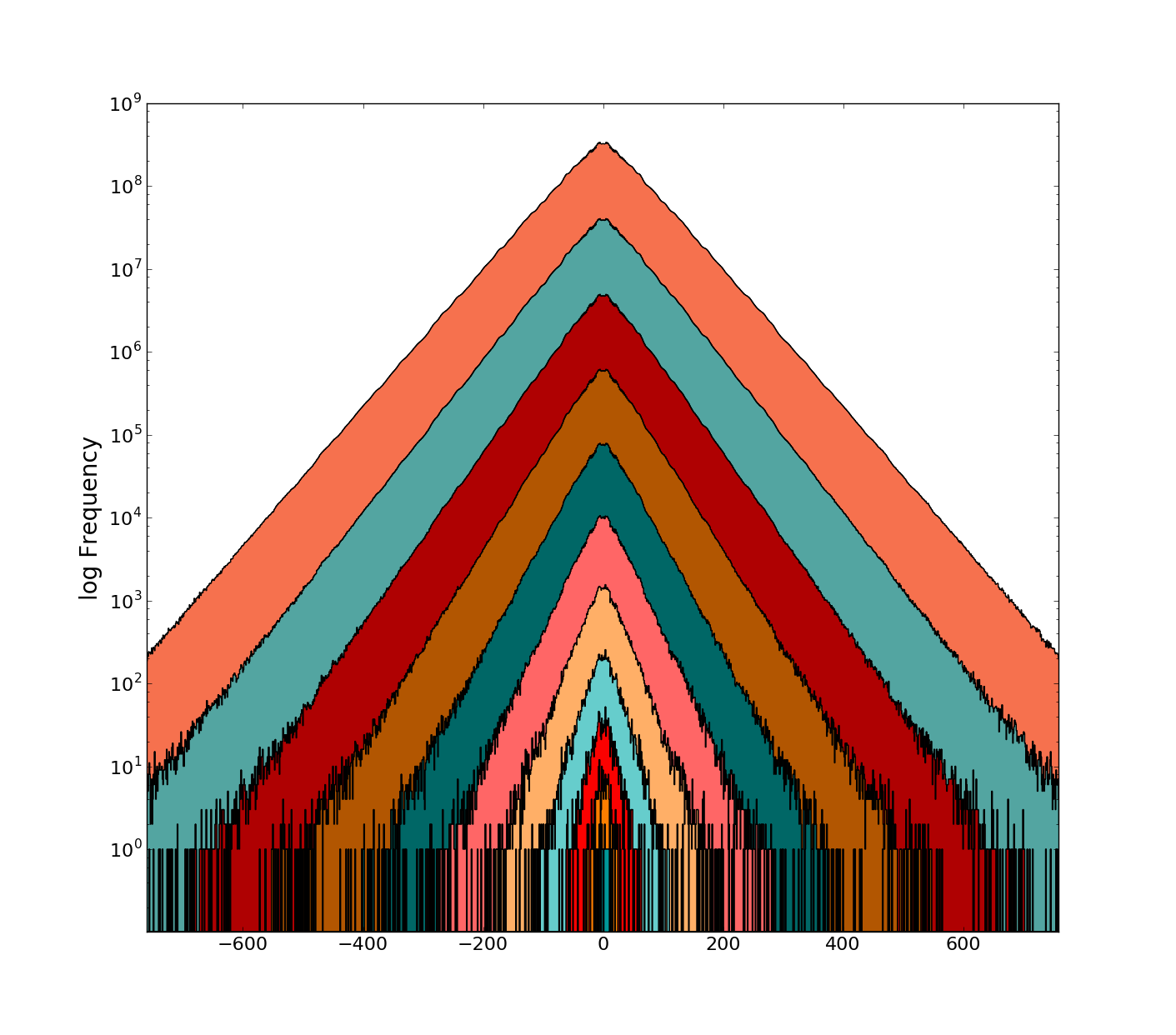}
		\caption{$\mathcal D_k^2=\mathcal D_{k+1}^1-\mathcal D_k^1$}
		\label{fig:Stacked_DD2_1e12}
	\end{subfigure}
	\begin{subfigure}[b]{0.48\textwidth}
		\includegraphics[width=\textwidth]{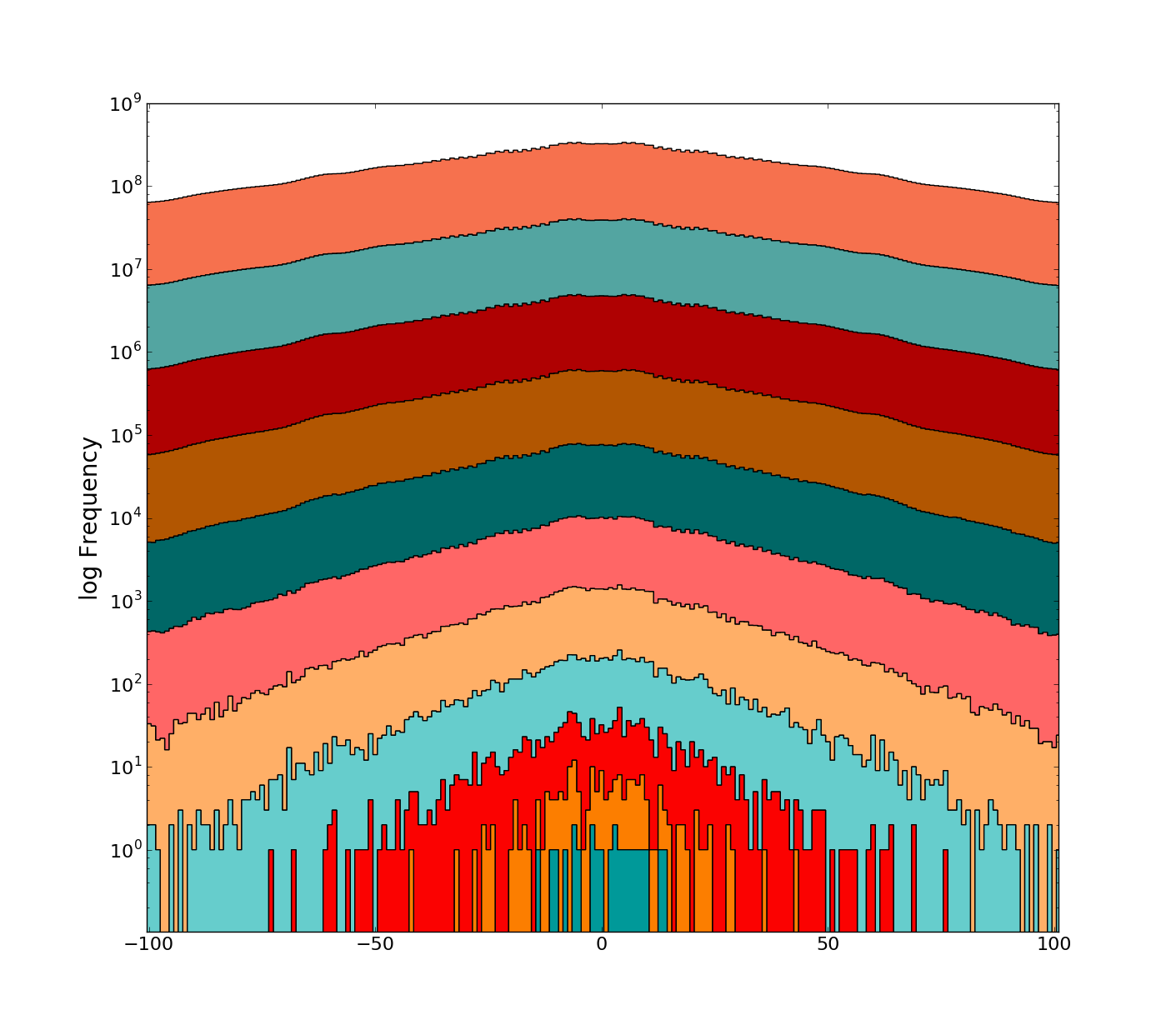}
		\caption{$\mathcal D_k^2$, close-up}
		\label{fig:Stacked_DD2_1e12_zoom}
	\end{subfigure}
	\caption{Stacked log-histograms of second order differences $\mathcal D_k^2$ between consecutive prime numbers on the number trail up to $N\leq 10^{12}$. }
	\label{fig:Stacked_DD2}
\end{figure}

\subsection*{Conclusion and open problems}

We introduced the prime grid to represent natural numbers with their unique prime signature. Using the $\ell_\infty$ norm for vectors, we defined the number trail, an analog of the number line. We studied the asymptotic growth and the distribution of primes on the number trail. We found that $L_{\infty}$ grows linearly at a specific rate $c_0$ based on our direct computations, which we supported by probabilistic methods. Furthermore, the prime gap functions of $\mathcal D^1$ and of $\mathcal D^2$ show interesting structure substantially different from the structure observed of $D^1$ and $D^2$. Our main aim was not necessarily to give precise proofs for all our claims, but rather to draw the attention towards our empirical findings that we consider interesting.

We conclude with a few questions that could be of interest. Most importantly, can there be some interplay between this setting and the traditional one? Is the sequence $\boldsymbol{\Omega}$ truly pseudorandom and does $\mathcal F$ describe all the forbidden words of $\boldsymbol{\Omega}$?

A more general question, given a subshift of finite type is it possible to characterize the set of all stationary distributions of the Markov chains compatible with the shift? Given a fixed distribution on the state space, can we decide whether it is the stationary distribution of a compatible Markov chain? If so, are there better analytical or optimization methods to find the Markov chain than we presented?
\noindent{\bf Acknowledgement.} The authors thank J\'anos Pintz for some useful discussions.

\bibliographystyle{abbrv}
\bibliography{prime_biblio}

\appendix

\section{C code for generating random sequences}\label{sec:Ccode}


The C code below was used to generate random sequences completely devoid of forbidden words as described in Subsection \ref{subsec:CompSimu}. The array indices are zero-based and the arrays are all initialized to zero. Given an input distribution $\mathbf p$ (see Table \ref{table:SimulatedDistributions}) a series of $f[i][j]$ truncated distributions are generated where the $i$ and $j$ indices refer to the $\mathbf p[i]$ and $\mathbf p[j]$ values, respectively. $f[i][j]$ consists of the ($\mathbf p[i]$,\ldots, $\mathbf p[j]$) slice of the original $\mathbf p$ distribution and it is normalized. The algorithm will choose the appropriate $f[i][j]$ distribution to sample according to the models of Subsection \ref{subsec:CompSimu}. The sampling function \textsl{ take\_a\_sample\_from( )} utilizes the GNU Scientific Library \cite{GNUlibrary} to efficently sample discrete distributions via the \textsl{gsl\_ran\_discrete\_preproc(~)} and \textsl{gsl\_ran\_discrete(~)} library functions. The output distribution $\Pi$ is computed from the values of the bin populations stored in the $\mathbf{counts}[~]$ array (line \ref{counts}) and the random sequence itself is stored in array $\mathbf{rand\_model}[~]$ (line \ref{rndmod}). In our nomenclature the random sequence consists of characters, represented by non-negative integers, and words are comprised of any number of consecutive characters. As explained in \eqref{eq:def_ForbiddenWords} certain words are forbidden in the random model and our algorithm below eliminates all these forbidden words.

The basic tenet of the algorithm is as follows. An $\mathbf{index[~]}$ array (line \ref{index}) keeps track of the then current lengths of growing words comprised of sets of (0), (0, 1), (0, 1, 2), etc.,  characters. For any given $x$ drawn in line \ref{x} there are two types of words, one whose character set contains $x$ and the other whose character set does not contain $x$. For those that contain $x$, appending $x$ to these words makes them potentially forbidden words depending on the current length of the word. The $\mathbf{index[~]}$ array keeps track of the lengths of different types of words and the test in line \ref{test} determines whether appending $x$ would complete a forbidden word, in which case a new $x$ must be drawn (lines \ref{model1}, and \ref{model2}). The former corresponds to Model~1 and the latter corresponds to Model~2 discussed in Subsection \ref{subsec:CompSimu}. When an appropriate $x$ has been drawn, every word type that does not contain the character $x$ can be reset by setting their $\mathbf{index[~]}$ array counter value to zero (line \ref{reset}). This means that for these word types, potentially forbidden words start growing again from scratch. Finally, the $\mathbf{counts[~]}$ array and the random model itself are both updated, and the new value of the largest character allowed by the current sample size, $\mathbf{max\_char}$ is determined (lines \ref{counts}, \ref{rndmod} and \ref{maxch}, respectively ).

\lstset{
	language=C,
	showstringspaces=false,
	keywordstyle=\color{black}\bfseries,
	commentstyle=\itshape,
	basicstyle={\ttfamily,\small},
	stringstyle=\slshape,
	numbers=left, numberstyle=\tiny,
	emph={floor, log, take_a_sample_from},
	emphstyle=\textbf,
	breaklines,
	breakatwhitespace,
	breakindent=20pt,
	escapeinside={(*@}{@*)}
}
\begin{lstlisting}
for( i=0, max_char=0; i<n_samples; ) {

/* take a sample, x is bin index, 0 <= x <= max_char */
x = take_a_sample_from( f[0][max_char] );   (*@\label{x}@*) 

/* ignore forbidden words longer than current sample */
for( k=max_char; k>=x; k-- ) {

/* for word type >= x, 
x extends forbidden word in the making, increment index */
++index[k];    (*@\label{index}@*)                       

/* forbidden word (for type k) appeared, fix last character */
if( index[k] >= (1 << (k+2)) ) {  /* 2^(k+2) */  (*@\label{test}@*)

/* try a new x */
x = k+1 + take_a_sample_from( f[k+1][max_char] );  /* MODEL 1 */ (*@\label{model1}@*)
//x = k+1;                                           /* MODEL 2 */  (*@\label{model2}@*)
break;
}
}

/* for word type < x,
x terminates forbidden word in the making */
for( k=0; k<x; k++ ) {
index[k] = 0;  (*@\label{reset}@*)
}

++counts[x];        /* incerement bin population */  (*@\label{counts}@*)
rand_model[i] = x;  /* append new x char to random model */  (*@\label{rndmod}@*)

/* max allowed char in current sample */
max_char = (int)(floor(log((double)(( ++i ) +1)) /log(2.)) -1.);  (*@\label{maxch}@*)
}
\end{lstlisting}

\section{Tabulated histograms}\label{sec:Hist_data}
\begin{footnotesize}
\begin{center}
\begin{longtable}{r|rrrrrrrrrrr}
\captionsetup{justification=centering}
\caption{Numerical data of  histograms of  $\mathcal D^1$ differences in column one for $N\leq10^{2-12}$ tabulated in columns 2-12}\label{tab:Hist_DD1_1e2-12}\\
\toprule
\multicolumn{1}{c|}{$\mathcal D^1$}  & \multicolumn{1}{c}{$10^{12}$}  & \multicolumn{1}{c}{$10^{11}$}  & \multicolumn{1}{c}{$10^{10}$}  & \multicolumn{1}{c}{$10^{9}$}  & \multicolumn{1}{c}{$10^{8}$}  & \multicolumn{1}{c}{$10^{7}$}  & \multicolumn{1}{c}{$10^{6}$}  & \multicolumn{1}{c}{$10^{5}$}  & \multicolumn{1}{c}{$10^{4}$}  & \multicolumn{1}{c}{$10^{3}$}  & \multicolumn{1}{c}{$10^{2}$}  \\ \midrule
\endfirsthead
\multicolumn{12}{c}{\tablename\ \thetable\ -- \textit{Continued from previous page}} \\ \toprule
\multicolumn{1}{c|}{$\mathcal D^1$}  & \multicolumn{1}{c}{$10^{12}$}  & \multicolumn{1}{c}{$10^{11}$}  & \multicolumn{1}{c}{$10^{10}$}  & \multicolumn{1}{c}{$10^{9}$}  & \multicolumn{1}{c}{$10^{8}$}  & \multicolumn{1}{c}{$10^{7}$}  & \multicolumn{1}{c}{$10^{6}$}  & \multicolumn{1}{c}{$10^{5}$}  & \multicolumn{1}{c}{$10^{4}$}  & \multicolumn{1}{c}{$10^{3}$}  & \multicolumn{1}{c}{$10^{2}$}  \\ \midrule
\endhead
\multicolumn{12}{r}{\textit{\small{Continued on next page}}} \\
\endfoot
\endlastfoot
\rowcolor[HTML]{E8E8E8} 
1  & 1         & 1        & 1        & 1       & 1      & 1     & 1    & 1   & 1  & 1  & 1 \\
2  & 545836959 & 65482891 & 8000621  & 999952  & 128678 & 17238 & 2352 & 358 & 63 & 12 & 3 \\
\rowcolor[HTML]{E8E8E8} 
3  & 0         & 0        & 0        & 0       & 0      & 0     & 0    & 0   & 0  & 0  & 0 \\
4  & 676962502 & 81197873 & 9918050  & 1238662 & 159572 & 21273 & 2962 & 460 & 83 & 14 & 4 \\
\rowcolor[HTML]{E8E8E8} 
5  & 0         & 0        & 0        & 0       & 0      & 0     & 0    & 0   & 0  & 0  & 0 \\
6  & 812569885 & 97460711 & 11910458 & 1489205 & 191644 & 25567 & 3508 & 519 & 86 & 17 & 3 \\
\rowcolor[HTML]{E8E8E8} 
7  & 233485333 & 28008641 & 3421613  & 427267  & 54555  & 7218  & 1024 & 151 & 24 & 3  & 1 \\
8  & 675924191 & 80872118 & 9851261  & 1225160 & 156491 & 20793 & 2882 & 439 & 70 & 13 & 1 \\
\rowcolor[HTML]{E8E8E8} 
9  & 382285674 & 45657562 & 5551012  & 689857  & 88083  & 11685 & 1644 & 241 & 34 & 10 & 2 \\
10 & 599704882 & 71643947 & 8701213  & 1078829 & 137483 & 18181 & 2481 & 353 & 52 & 9  & 1 \\
\rowcolor[HTML]{E8E8E8} 
11 & 275043047 & 32785324 & 3977086  & 492050  & 62471  & 8189  & 1152 & 163 & 30 & 0  & 0 \\
12 & 756641841 & 90135360 & 10920595 & 1351869 & 171761 & 22303 & 3026 & 420 & 63 & 4  & 1 \\
\rowcolor[HTML]{E8E8E8} 
13 & 384294766 & 45758060 & 5540132  & 683669  & 86297  & 11231 & 1547 & 239 & 30 & 6  & 1 \\
14 & 677173536 & 80473232 & 9718049  & 1196206 & 150599 & 19575 & 2670 & 369 & 63 & 9  & 1 \\
\rowcolor[HTML]{E8E8E8} 
15 & 401018460 & 47543338 & 5723867  & 703177  & 88375  & 11521 & 1513 & 237 & 37 & 5  & 0 \\
16 & 620181299 & 73474151 & 8841576  & 1083654 & 135551 & 17349 & 2271 & 311 & 49 & 13 & 3 \\
\rowcolor[HTML]{E8E8E8} 
17 & 424850184 & 50195576 & 6020340  & 735419  & 91699  & 11709 & 1583 & 224 & 27 & 4  & 1 \\
18 & 597556733 & 70592210 & 8465897  & 1033639 & 128604 & 16478 & 2171 & 293 & 39 & 3  & 0 \\
\rowcolor[HTML]{E8E8E8} 
19 & 417182685 & 49185886 & 5882046  & 715767  & 88879  & 11340 & 1491 & 199 & 22 & 6  & 0 \\
20 & 500276741 & 58952639 & 7052759  & 857727  & 106288 & 13622 & 1822 & 248 & 35 & 8  & 1 \\
\rowcolor[HTML]{E8E8E8} 
21 & 407144642 & 47774244 & 5683277  & 686540  & 84352  & 10560 & 1395 & 181 & 27 & 5  &   \\
22 & 484516599 & 56852120 & 6763219  & 816597  & 100701 & 12564 & 1623 & 198 & 28 & 3  &   \\
\rowcolor[HTML]{E8E8E8} 
23 & 443439269 & 51900616 & 6153871  & 740041  & 90225  & 11240 & 1434 & 199 & 19 & 1  &   \\
24 & 470479351 & 55067706 & 6528288  & 784364  & 96024  & 12067 & 1551 & 198 & 26 & 1  &   \\
\rowcolor[HTML]{E8E8E8} 
25 & 486683648 & 56773666 & 6700619  & 802366  & 97486  & 11964 & 1418 & 158 & 18 & 2  &   \\
26 & 492742414 & 57478966 & 6784519  & 812971  & 98737  & 12325 & 1526 & 196 & 30 & 3  &   \\
\rowcolor[HTML]{E8E8E8} 
27 & 447830032 & 52126930 & 6134631  & 731493  & 88573  & 11033 & 1263 & 142 & 19 & 1  &   \\
28 & 415327048 & 48370523 & 5701776  & 680946  & 82323  & 10102 & 1263 & 165 & 19 & 2  &   \\
\rowcolor[HTML]{E8E8E8} 
29 & 390303895 & 45329992 & 5320982  & 630812  & 76393  & 9259  & 1153 & 141 & 14 & 1  &   \\
30 & 392437767 & 45575573 & 5347835  & 635587  & 76130  & 9320  & 1167 & 133 & 21 & 3  &   \\
\rowcolor[HTML]{E8E8E8} 
31 & 411121720 & 47566224 & 5555827  & 655573  & 78270  & 9417  & 1109 & 145 & 15 & 0  &   \\
32 & 391706876 & 45315055 & 5292173  & 625088  & 74690  & 8998  & 1063 & 119 & 8  & 2  &   \\
\rowcolor[HTML]{E8E8E8} 
33 & 407916113 & 47009979 & 5467687  & 641287  & 75646  & 8971  & 1057 & 117 & 10 & 1  &   \\
34 & 357435898 & 41178422 & 4788218  & 561466  & 66411  & 7902  & 969  & 127 & 15 & 0  &   \\
\rowcolor[HTML]{E8E8E8} 
35 & 415666860 & 47710551 & 5519737  & 643360  & 75712  & 9030  & 1089 & 120 & 9  & 2  &   \\
36 & 357761163 & 41026033 & 4741553  & 552125  & 64897  & 7588  & 901  & 95  & 6  & 1  &   \\
\rowcolor[HTML]{E8E8E8} 
37 & 438008056 & 50053582 & 5759940  & 666505  & 77507  & 8914  & 1032 & 109 & 7  & 0  &   \\
38 & 362734703 & 41403690 & 4757877  & 548943  & 63884  & 7324  & 856  & 110 & 11 & 0  &   \\
\rowcolor[HTML]{E8E8E8} 
39 & 419338332 & 47753944 & 5473818  & 631038  & 72661  & 8342  & 906  & 88  & 12 & 0  &   \\
40 & 319372932 & 36349710 & 4161936  & 478868  & 55499  & 6364  & 687  & 72  & 9  & 1  &   \\
\rowcolor[HTML]{E8E8E8} 
41 & 384186364 & 43587533 & 4972936  & 570258  & 65417  & 7486  & 844  & 92  & 6  & 0  &   \\
42 & 310770326 & 35210712 & 4008763  & 458211  & 52511  & 5895  & 656  & 69  & 6  & 0  &   \\
\rowcolor[HTML]{E8E8E8} 
43 & 375151082 & 42408631 & 4809769  & 548333  & 62368  & 7088  & 800  & 83  & 6  & 0  &   \\
44 & 301328314 & 34025495 & 3854173  & 438680  & 50067  & 5531  & 665  & 72  & 6  & 0  &   \\
\rowcolor[HTML]{E8E8E8} 
45 & 342858056 & 38657690 & 4375235  & 495639  & 55826  & 6342  & 710  & 74  & 7  & 0  &   \\
46 & 274048457 & 30861067 & 3483910  & 394841  & 44666  & 4925  & 534  & 58  & 4  & 0  &   \\
\rowcolor[HTML]{E8E8E8} 
47 & 314323046 & 35331031 & 3980879  & 448389  & 50260  & 5599  & 625  & 65  & 5  & 0  &   \\
48 & 268461025 & 30128434 & 3390308  & 381310  & 42656  & 4740  & 519  & 50  & 2  & 0  &   \\
\rowcolor[HTML]{E8E8E8} 
49 & 313748331 & 35145253 & 3944099  & 442484  & 49301  & 5395  & 602  & 67  & 4  & 1  &   \\
50 & 279059340 & 31197050 & 3493274  & 390264  & 43288  & 4677  & 502  & 43  & 3  &    &   \\
\rowcolor[HTML]{E8E8E8} 
51 & 316192097 & 35270467 & 3938108  & 438494  & 48632  & 5216  & 555  & 54  & 2  &    &   \\
52 & 290982321 & 32393863 & 3612103  & 401777  & 44226  & 4687  & 490  & 39  & 4  &    &   \\
\rowcolor[HTML]{E8E8E8} 
53 & 321173394 & 35704308 & 3971568  & 440735  & 48442  & 5238  & 546  & 46  & 4  &    &   \\
54 & 296555755 & 32916809 & 3654866  & 404294  & 44375  & 4865  & 487  & 50  & 1  &    &   \\
\rowcolor[HTML]{E8E8E8} 
55 & 310144069 & 34377195 & 3805054  & 419457  & 46030  & 4780  & 485  & 38  & 4  &    &   \\
56 & 278244969 & 30786869 & 3402596  & 375360  & 40765  & 4250  & 377  & 23  & 1  &    &   \\
\rowcolor[HTML]{E8E8E8} 
57 & 274938973 & 30365686 & 3349189  & 367275  & 39855  & 4219  & 452  & 36  & 5  &    &   \\
58 & 251238297 & 27693610 & 3046692  & 332896  & 35716  & 3799  & 390  & 23  & 3  &    &   \\
\rowcolor[HTML]{E8E8E8} 
59 & 255057585 & 28064094 & 3078949  & 334786  & 35881  & 3707  & 364  & 26  & 1  &    &   \\
60 & 240625481 & 26406855 & 2893001  & 314572  & 33529  & 3453  & 356  & 27  & 4  &    &   \\
\rowcolor[HTML]{E8E8E8} 
61 & 240098350 & 26312891 & 2874685  & 311734  & 33303  & 3428  & 317  & 25  & 0  &    &   \\
62 & 230501322 & 25221244 & 2747421  & 296466  & 31590  & 3203  & 317  & 32  & 2  &    &   \\
\rowcolor[HTML]{E8E8E8} 
63 & 240699750 & 26287308 & 2858713  & 308409  & 32549  & 3323  & 281  & 17  & 1  &    &   \\
64 & 245034155 & 26679075 & 2893017  & 309735  & 32703  & 3363  & 310  & 34  & 1  &    &   \\
\rowcolor[HTML]{E8E8E8} 
65 & 261624047 & 28405280 & 3071254  & 328501  & 34295  & 3350  & 309  & 32  & 3  &    &   \\
66 & 266059296 & 28835866 & 3106800  & 331091  & 34325  & 3517  & 320  & 23  & 1  &    &   \\
\rowcolor[HTML]{E8E8E8} 
67 & 276300134 & 29870623 & 3211217  & 341332  & 35612  & 3414  & 301  & 24  & 1  &    &   \\
68 & 271214452 & 29261959 & 3136401  & 331350  & 34448  & 3383  & 321  & 20  & 0  &    &   \\
\rowcolor[HTML]{E8E8E8} 
69 & 269383864 & 29009679 & 3101086  & 327085  & 33620  & 3299  & 272  & 13  & 0  &    &   \\
70 & 254946493 & 27403933 & 2922126  & 306593  & 31439  & 3079  & 282  & 23  & 2  &    &   \\
\rowcolor[HTML]{E8E8E8} 
71 & 243206891 & 26090253 & 2773428  & 290862  & 29670  & 3004  & 241  & 21  & 1  &    &   \\
72 & 225296489 & 24122260 & 2558040  & 267104  & 27016  & 2579  & 232  & 14  & 2  &    &   \\
\rowcolor[HTML]{E8E8E8} 
73 & 208369696 & 22276198 & 2361046  & 246199  & 24927  & 2467  & 223  & 20  & 0  &    &   \\
74 & 196338136 & 20936263 & 2209352  & 229402  & 23131  & 2200  & 198  & 12  & 0  &    &   \\
\rowcolor[HTML]{E8E8E8} 
75 & 188391164 & 20046979 & 2110189  & 218518  & 21830  & 1943  & 156  & 10  & 0  &    &   \\
76 & 186442645 & 19783796 & 2074672  & 213678  & 21235  & 1923  & 166  & 10  & 2  &    &   \\
\rowcolor[HTML]{E8E8E8} 
77 & 184682531 & 19558506 & 2046017  & 208819  & 20672  & 1945  & 179  & 12  & 1  &    &   \\
78 & 183172563 & 19350290 & 2022416  & 207385  & 20373  & 1881  & 148  & 7   & 0  &    &   \\
\rowcolor[HTML]{E8E8E8} 
79 & 178996142 & 18878770 & 1964451  & 200764  & 19730  & 1845  & 135  & 9   & 0  &    &   \\
80 & 172704652 & 18184358 & 1890852  & 192546  & 18867  & 1728  & 138  & 7   & 0  &    &   \\ \bottomrule
\end{longtable}
\end{center}
\end{footnotesize}

\begin{footnotesize}
\begin{center}
\begin{longtable}{r|rrrrrrrrrrr}
\captionsetup{justification=centering}
\caption{Numerical data of  histograms of  $\mathcal D^2$ differences in column one for $N\leq10^{2-12}$ tabulated in columns 2-12}\label{tab:Hist_DD2_1e2-12}\\
\toprule
\multicolumn{1}{c|}{$\mathcal D^2$}  & \multicolumn{1}{c}{$10^{12}$}  & \multicolumn{1}{c}{$10^{11}$}  & \multicolumn{1}{c}{$10^{10}$}  & \multicolumn{1}{c}{$10^{9}$}  & \multicolumn{1}{c}{$10^{8}$}  & \multicolumn{1}{c}{$10^{7}$}  & \multicolumn{1}{c}{$10^{6}$}  & \multicolumn{1}{c}{$10^{5}$}  & \multicolumn{1}{c}{$10^{4}$}  & \multicolumn{1}{c}{$10^{3}$}  & \multicolumn{1}{c}{$10^{2}$}  \\ \midrule
\endfirsthead
\multicolumn{12}{c}{\tablename\ \thetable\ -- \textit{Continued from previous page}} \\ \toprule
\multicolumn{1}{c|}{$\mathcal D^2$}  & \multicolumn{1}{c}{$10^{12}$}  & \multicolumn{1}{c}{$10^{11}$}  &  \multicolumn{1}{c}{$10^{10}$}  & \multicolumn{1}{c}{$10^{9}$}  & \multicolumn{1}{c}{$10^{8}$}  & \multicolumn{1}{c}{$10^{7}$}  & \multicolumn{1}{c}{$10^{6}$}  & \multicolumn{1}{c}{$10^{5}$}  & \multicolumn{1}{c}{$10^{4}$}  & \multicolumn{1}{c}{$10^{3}$}  & \multicolumn{1}{c}{$10^{2}$}  \\ \midrule
\endhead
\multicolumn{12}{r}{\textit{\small{Continued on next page}}} \\
\endfoot
\endlastfoot
\rowcolor[HTML]{E8E8E8} 
-60 & 124678820 & 13617925 & 1479775 & 159214 & 16731 & 1719 & 140  & 9   & 2  &    &   \\
-59 & 125417281 & 13698994 & 1493730 & 161701 & 17398 & 1755 & 162  & 20  & 3  &    &   \\
\rowcolor[HTML]{E8E8E8} 
-58 & 125911011 & 13764630 & 1500292 & 162002 & 17011 & 1657 & 172  & 18  & 0  &    &   \\
-57 & 126923083 & 13894104 & 1516473 & 164037 & 17363 & 1785 & 182  & 18  & 0  &    &   \\
\rowcolor[HTML]{E8E8E8} 
-56 & 128904390 & 14133007 & 1547182 & 168320 & 18180 & 1859 & 172  & 20  & 1  &    &   \\
-55 & 131234765 & 14416090 & 1578825 & 173076 & 18466 & 1934 & 185  & 14  & 0  &    &   \\
\rowcolor[HTML]{E8E8E8} 
-54 & 134222294 & 14782109 & 1626127 & 177965 & 19312 & 2032 & 193  & 16  & 1  &    &   \\
-53 & 137524028 & 15185420 & 1676289 & 183678 & 19652 & 2153 & 231  & 15  & 1  &    &   \\
\rowcolor[HTML]{E8E8E8} 
-52 & 140779422 & 15571660 & 1718999 & 188928 & 20752 & 2206 & 201  & 11  & 1  &    &   \\
-51 & 144402779 & 16017844 & 1776828 & 196358 & 21781 & 2337 & 214  & 21  & 4  &    &   \\
\rowcolor[HTML]{E8E8E8} 
-50 & 147259221 & 16369755 & 1820487 & 201629 & 21894 & 2394 & 240  & 14  & 0  &    &   \\
-49 & 150647376 & 16780566 & 1870272 & 208691 & 23276 & 2460 & 258  & 21  & 1  &    &   \\
\rowcolor[HTML]{E8E8E8} 
-48 & 152336894 & 16989890 & 1895485 & 210985 & 23352 & 2546 & 266  & 22  & 2  &    &   \\
-47 & 154896930 & 17306325 & 1936930 & 217237 & 24137 & 2622 & 272  & 30  & 1  &    &   \\
\rowcolor[HTML]{E8E8E8} 
-46 & 154945318 & 17317334 & 1937210 & 216378 & 23987 & 2600 & 277  & 23  & 1  &    &   \\
-45 & 157288655 & 17615557 & 1976344 & 221746 & 24731 & 2679 & 281  & 23  & 4  &    &   \\
\rowcolor[HTML]{E8E8E8} 
-44 & 157474057 & 17640425 & 1980331 & 222635 & 24737 & 2670 & 265  & 25  & 1  &    &   \\
-43 & 161440665 & 18131460 & 2041299 & 230497 & 25939 & 2801 & 319  & 34  & 3  &    &   \\
\rowcolor[HTML]{E8E8E8} 
-42 & 161477126 & 18139987 & 2042247 & 230677 & 26217 & 2963 & 323  & 35  & 4  & 1  &   \\
-41 & 166326919 & 18758031 & 2124063 & 240927 & 27317 & 3107 & 338  & 41  & 5  & 0  &   \\
\rowcolor[HTML]{E8E8E8} 
-40 & 166066810 & 18715681 & 2117644 & 240503 & 27342 & 2987 & 352  & 38  & 2  & 0  &   \\
-39 & 172416970 & 19515444 & 2219822 & 253386 & 28718 & 3232 & 321  & 34  & 2  & 0  &   \\
\rowcolor[HTML]{E8E8E8} 
-38 & 172061911 & 19471804 & 2214742 & 252503 & 28497 & 3286 & 363  & 37  & 3  & 0  &   \\
-37 & 179636293 & 20424829 & 2335566 & 268516 & 30951 & 3562 & 378  & 45  & 2  & 0  &   \\
\rowcolor[HTML]{E8E8E8} 
-36 & 177640738 & 20172741 & 2304548 & 264324 & 30438 & 3401 & 382  & 42  & 4  & 0  &   \\
-35 & 186402059 & 21262299 & 2443733 & 282221 & 32943 & 3845 & 429  & 56  & 7  & 0  &   \\
\rowcolor[HTML]{E8E8E8} 
-34 & 183557880 & 20917244 & 2399966 & 277126 & 32006 & 3774 & 441  & 50  & 3  & 0  &   \\
-33 & 192595098 & 22050424 & 2542567 & 295787 & 34423 & 4026 & 458  & 50  & 6  & 0  &   \\
\rowcolor[HTML]{E8E8E8} 
-32 & 187933156 & 21473919 & 2470177 & 286212 & 33446 & 3806 & 454  & 56  & 7  & 1  &   \\
-31 & 197513541 & 22680026 & 2625029 & 305443 & 35984 & 4232 & 477  & 56  & 7  & 0  &   \\
\rowcolor[HTML]{E8E8E8} 
-30 & 192806659 & 22085709 & 2551227 & 296524 & 35018 & 4006 & 466  & 46  & 7  & 0  &   \\
-29 & 200679723 & 23085997 & 2679220 & 313792 & 37043 & 4416 & 538  & 61  & 6  & 0  &   \\
\rowcolor[HTML]{E8E8E8} 
-28 & 196747898 & 22596987 & 2618029 & 306978 & 35944 & 4255 & 526  & 68  & 14 & 1  &   \\
-27 & 208373342 & 24066755 & 2807751 & 331861 & 39362 & 4796 & 620  & 66  & 6  & 0  &   \\
\rowcolor[HTML]{E8E8E8} 
-26 & 207524517 & 23975103 & 2798797 & 330124 & 39711 & 4907 & 595  & 71  & 9  & 2  &   \\
-25 & 220660508 & 25634863 & 3011643 & 358793 & 43228 & 5235 & 652  & 72  & 12 & 1  &   \\
\rowcolor[HTML]{E8E8E8} 
-24 & 218679154 & 25398221 & 2985061 & 355705 & 42957 & 5207 & 683  & 95  & 13 & 2  &   \\
-23 & 231838351 & 27072426 & 3201247 & 383263 & 46861 & 5759 & 698  & 93  & 10 & 0  &   \\
\rowcolor[HTML]{E8E8E8} 
-22 & 227001776 & 26473195 & 3125624 & 374095 & 45595 & 5669 & 717  & 72  & 8  & 0  &   \\
-21 & 238581530 & 27942303 & 3315655 & 399252 & 49078 & 6132 & 772  & 93  & 9  & 0  &   \\
\rowcolor[HTML]{E8E8E8} 
-20 & 225379188 & 26275304 & 3104488 & 371187 & 45314 & 5658 & 734  & 101 & 12 & 1  &   \\
-19 & 238927515 & 27983835 & 3319062 & 400185 & 49312 & 6217 & 745  & 98  & 12 & 4  &   \\
\rowcolor[HTML]{E8E8E8} 
-18 & 228438575 & 26674122 & 3153715 & 379285 & 46182 & 5764 & 753  & 99  & 14 & 1  &   \\
-17 & 243203220 & 28525498 & 3391278 & 410922 & 50370 & 6350 & 821  & 124 & 21 & 2  &   \\
\rowcolor[HTML]{E8E8E8} 
-16 & 235237016 & 27535204 & 3268627 & 393297 & 48604 & 6120 & 770  & 113 & 20 & 1  &   \\
-15 & 251141295 & 29529930 & 3521444 & 425681 & 52563 & 6507 & 855  & 108 & 13 & 0  &   \\
\rowcolor[HTML]{E8E8E8} 
-14 & 245657729 & 28854709 & 3434713 & 415428 & 51396 & 6518 & 811  & 115 & 17 & 3  & 1 \\
-13 & 261697331 & 30871436 & 3694997 & 449396 & 56052 & 7063 & 941  & 133 & 16 & 1  & 0 \\
\rowcolor[HTML]{E8E8E8} 
-12 & 253325086 & 29845788 & 3566511 & 433328 & 53567 & 6944 & 953  & 137 & 18 & 5  & 0 \\
-11 & 273796317 & 32442924 & 3903869 & 479799 & 59764 & 7739 & 1018 & 146 & 16 & 3  & 1 \\
\rowcolor[HTML]{E8E8E8} 
-10 & 274354895 & 32570409 & 3933248 & 483817 & 61117 & 7859 & 1090 & 158 & 22 & 3  & 1 \\
-9  & 286644016 & 34099599 & 4120679 & 508154 & 64219 & 8408 & 1124 & 162 & 25 & 5  & 0 \\
\rowcolor[HTML]{E8E8E8} 
-8  & 285492760 & 34048670 & 4132705 & 512306 & 65401 & 8539 & 1218 & 157 & 30 & 3  & 1 \\
-7  & 294625344 & 35172331 & 4274938 & 530226 & 67242 & 8802 & 1238 & 178 & 36 & 10 & 0 \\
\rowcolor[HTML]{E8E8E8} 
-6  & 288344140 & 34440842 & 4187603 & 519919 & 66531 & 8756 & 1258 & 179 & 32 & 10 & 2 \\
-5  & 296011555 & 35390375 & 4307092 & 534847 & 68057 & 9124 & 1258 & 161 & 29 & 4  & 1 \\
\rowcolor[HTML]{E8E8E8} 
-4  & 286503794 & 34243792 & 4167061 & 518758 & 66184 & 8826 & 1225 & 179 & 22 & 1  & 0 \\
-3  & 284220703 & 33864884 & 4102245 & 508054 & 64912 & 8449 & 1142 & 167 & 18 & 3  & 0 \\
\rowcolor[HTML]{E8E8E8} 
-2  & 282984904 & 33788571 & 4106387 & 509356 & 65280 & 8448 & 1188 & 181 & 28 & 8  & 2 \\
-1  & 285790536 & 34071252 & 4131348 & 511701 & 65077 & 8515 & 1203 & 163 & 20 & 4  & 1 \\
\rowcolor[HTML]{E8E8E8} 
0   & 285132051 & 34090579 & 4150019 & 517031 & 66320 & 8805 & 1190 & 170 & 23 & 8  & 1 \\
1   & 285816973 & 34079474 & 4134226 & 512001 & 64923 & 8454 & 1166 & 180 & 22 & 4  & 0 \\
\rowcolor[HTML]{E8E8E8} 
2   & 282965478 & 33790755 & 4105123 & 510733 & 65120 & 8741 & 1243 & 164 & 24 & 4  & 1 \\
3   & 284247018 & 33856898 & 4102916 & 508198 & 64138 & 8292 & 1171 & 182 & 29 & 5  & 2 \\
\rowcolor[HTML]{E8E8E8} 
4   & 286532237 & 34254862 & 4168340 & 517632 & 66368 & 8872 & 1297 & 203 & 44 & 7  & 1 \\
5   & 296029198 & 35386325 & 4304568 & 534186 & 68239 & 8887 & 1188 & 172 & 19 & 3  & 1 \\
\rowcolor[HTML]{E8E8E8} 
6   & 288364738 & 34451448 & 4192449 & 522034 & 66343 & 8860 & 1169 & 166 & 29 & 6  & 1 \\
7   & 294642143 & 35177434 & 4275015 & 530500 & 67604 & 8935 & 1228 & 168 & 24 & 6  & 1 \\
\rowcolor[HTML]{E8E8E8} 
8   & 285494550 & 34054118 & 4132141 & 513234 & 65554 & 8650 & 1166 & 152 & 27 & 5  & 1 \\
9   & 286594727 & 34088658 & 4124488 & 508351 & 64462 & 8372 & 1148 & 169 & 30 & 7  & 1 \\
\rowcolor[HTML]{E8E8E8} 
10  & 274348420 & 32567912 & 3931836 & 484093 & 61559 & 7998 & 1120 & 152 & 26 & 3  & 1 \\
11  & 273777740 & 32426966 & 3899784 & 477812 & 60258 & 7824 & 1110 & 164 & 20 & 1  & 0 \\
\rowcolor[HTML]{E8E8E8} 
12  & 253311421 & 29839307 & 3566489 & 433427 & 53589 & 6795 & 820  & 108 & 12 & 1  & 0 \\
13  & 261679280 & 30866198 & 3694130 & 450241 & 55652 & 6897 & 898  & 124 & 24 & 5  & 1 \\
\rowcolor[HTML]{E8E8E8} 
14  & 245666156 & 28870503 & 3440583 & 416527 & 51179 & 6679 & 894  & 129 & 22 & 2  & 1 \\
15  & 251104118 & 29523322 & 3520513 & 426857 & 52579 & 6798 & 862  & 110 & 15 & 2  &   \\
\rowcolor[HTML]{E8E8E8} 
16  & 235222471 & 27532894 & 3263211 & 392940 & 48632 & 6059 & 786  & 93  & 9  & 0  &   \\
17  & 243217035 & 28529132 & 3394144 & 409862 & 50337 & 6288 & 822  & 101 & 18 & 2  &   \\
\rowcolor[HTML]{E8E8E8} 
18  & 228433005 & 26670900 & 3154649 & 378424 & 46025 & 5816 & 738  & 97  & 14 & 2  &   \\
19  & 238932915 & 27978916 & 3321718 & 399985 & 49124 & 6348 & 795  & 101 & 9  & 1  &   \\
\rowcolor[HTML]{E8E8E8} 
20  & 225357244 & 26282933 & 3105430 & 372688 & 45379 & 5720 & 683  & 92  & 20 & 0  &   \\
21  & 238581056 & 27934903 & 3314769 & 398874 & 49033 & 6267 & 801  & 106 & 10 & 3  &   \\
\rowcolor[HTML]{E8E8E8} 
22  & 227016320 & 26480081 & 3126879 & 374803 & 45568 & 5598 & 702  & 115 & 15 & 1  &   \\
23  & 231864937 & 27070736 & 3199074 & 383370 & 46714 & 5774 & 756  & 103 & 9  & 1  &   \\
\rowcolor[HTML]{E8E8E8} 
24  & 218674999 & 25405730 & 2986737 & 356172 & 43085 & 5357 & 638  & 78  & 10 & 2  &   \\
25  & 220654476 & 25636190 & 3013633 & 358464 & 43490 & 5245 & 656  & 63  & 13 & 0  &   \\
\rowcolor[HTML]{E8E8E8} 
26  & 207550224 & 23975565 & 2797734 & 329791 & 39477 & 4686 & 556  & 66  & 6  & 0  &   \\
27  & 208353637 & 24063968 & 2807640 & 330900 & 39448 & 4718 & 613  & 74  & 10 & 0  &   \\
\rowcolor[HTML]{E8E8E8} 
28  & 196716398 & 22593961 & 2621784 & 306740 & 36422 & 4337 & 502  & 49  & 6  & 2  &   \\
29  & 200688403 & 23081340 & 2678062 & 313512 & 37251 & 4541 & 539  & 79  & 9  & 1  &   \\
\rowcolor[HTML]{E8E8E8} 
30  & 192822269 & 22086684 & 2552125 & 297179 & 34594 & 4073 & 467  & 52  & 4  & 0  &   \\
31  & 197517587 & 22669608 & 2621551 & 305250 & 36005 & 4255 & 493  & 61  & 8  & 0  &   \\
\rowcolor[HTML]{E8E8E8} 
32  & 187932603 & 21464007 & 2468752 & 286987 & 33400 & 3951 & 494  & 60  & 4  & 0  &   \\
33  & 192596736 & 22046542 & 2542733 & 296132 & 34594 & 4077 & 458  & 48  & 1  & 0  &   \\
\rowcolor[HTML]{E8E8E8} 
34  & 183600374 & 20928566 & 2398187 & 275541 & 31768 & 3651 & 440  & 60  & 5  & 0  &   \\
35  & 186373340 & 21266926 & 2441600 & 282075 & 32839 & 3936 & 451  & 45  & 1  & 0  &   \\
\rowcolor[HTML]{E8E8E8} 
36  & 177654153 & 20160128 & 2302012 & 264274 & 30523 & 3503 & 392  & 45  & 6  & 1  &   \\
37  & 179653033 & 20422240 & 2334627 & 267943 & 30676 & 3620 & 403  & 36  & 4  & 0  &   \\
\rowcolor[HTML]{E8E8E8} 
38  & 172064515 & 19474489 & 2214595 & 252591 & 28745 & 3279 & 338  & 38  & 5  & 0  &   \\
39  & 172423860 & 19521029 & 2221306 & 253670 & 28972 & 3311 & 373  & 37  & 5  & 0  &   \\
\rowcolor[HTML]{E8E8E8} 
40  & 166104905 & 18724906 & 2117971 & 240611 & 27344 & 3091 & 371  & 43  & 3  & 0  &   \\
41  & 166323319 & 18749509 & 2123150 & 240448 & 27089 & 3053 & 338  & 47  & 4  & 0  &   \\
\rowcolor[HTML]{E8E8E8} 
42  & 161484723 & 18142330 & 2044048 & 230015 & 25900 & 2796 & 310  & 29  & 1  & 0  &   \\
43  & 161411134 & 18136861 & 2044972 & 230317 & 26167 & 2940 & 338  & 30  & 3  & 1  &   \\
\rowcolor[HTML]{E8E8E8} 
44  & 157504079 & 17645743 & 1979344 & 222051 & 24488 & 2675 & 286  & 27  & 3  &    &   \\
45  & 157289722 & 17615572 & 1977617 & 222128 & 24897 & 2697 & 278  & 27  & 1  &    &   \\
\rowcolor[HTML]{E8E8E8} 
46  & 154924957 & 17311332 & 1937595 & 216315 & 23849 & 2569 & 266  & 18  & 1  &    &   \\
47  & 154898752 & 17306039 & 1937480 & 216551 & 23822 & 2619 & 299  & 25  & 3  &    &   \\
\rowcolor[HTML]{E8E8E8} 
48  & 152317224 & 16983058 & 1896656 & 211348 & 23314 & 2505 & 236  & 34  & 3  &    &   \\
49  & 150644552 & 16780039 & 1867095 & 207637 & 22946 & 2471 & 238  & 21  & 3  &    &   \\
\rowcolor[HTML]{E8E8E8} 
50  & 147271752 & 16360151 & 1817438 & 200919 & 22107 & 2344 & 225  & 20  & 0  &    &   \\
51  & 144377149 & 16013063 & 1776781 & 195943 & 21361 & 2212 & 224  & 11  & 1  &    &   \\
\rowcolor[HTML]{E8E8E8} 
52  & 140821922 & 15572918 & 1723600 & 189431 & 20598 & 2164 & 207  & 23  & 0  &    &   \\
53  & 137514881 & 15179551 & 1672906 & 183860 & 20090 & 2121 & 213  & 23  & 1  &    &   \\
\rowcolor[HTML]{E8E8E8} 
54  & 134230989 & 14783335 & 1627558 & 177736 & 19311 & 2036 & 199  & 21  & 1  &    &   \\
55  & 131236967 & 14423710 & 1582500 & 172696 & 18457 & 1861 & 175  & 17  & 1  &    &   \\
\rowcolor[HTML]{E8E8E8} 
56  & 128906212 & 14137741 & 1546258 & 168321 & 17967 & 1829 & 186  & 13  & 1  &    &   \\
57  & 126912794 & 13892042 & 1514923 & 163999 & 17368 & 1781 & 178  & 14  & 0  &    &   \\
\rowcolor[HTML]{E8E8E8} 
58  & 125903469 & 13766302 & 1499470 & 161793 & 17046 & 1688 & 158  & 10  & 0  &    &   \\
59  & 125423037 & 13700814 & 1490810 & 160879 & 17254 & 1804 & 152  & 11  & 1  &    &   \\
\rowcolor[HTML]{E8E8E8} 
60  & 124678744 & 13607074 & 1477998 & 159281 & 16989 & 1674 & 156  & 22  & 2  &    &   \\ \bottomrule
\end{longtable}
\end{center}
\end{footnotesize}

\section{Sagemath Python code}\label{sec:Code}

Computations were carried out with \textsl{Sagemath} 7.2 \cite{SageMath} running on 8 hyperthreaded 3.6 GHz i7 CPU cores, and using 32 GB of RAM and about 500 GB of harddisk storage on a 64-bit Linux desktop computer. It took about 80 days of total walltime to complete all computations up to the order of $N=10^{12}$. Hardware resources were maximized by running the computations in 20 independent batches and utilizing a hand-tuned, load-balanced parallelization protocol. The code listing below is mostly self explanatory with comments and we provide detailed explanations where necessary, referring to line numbers.

(\ref{bignumpy}) The only non-standard package we used is called \textsl{bignumpy} \cite{Bignumpy} and it is crucial for seamlessly handling very large \textsl{numpy} arrays. Our computations heavily utilize \textsl{numpy} array operations but these become prohibitive using standard \textsl{numpy} when the size of a \textsl{numpy} array exhausts available memory. \textsl{Bignumpy} provides file backed \textsl{numpy} array objects for \textsl{Python} utilizing the mmap/shared memory feature of Unix. \textsl{Bignumpy} allows for streamlined array operations on arrays multiple times larger than available memory, with a negligible slowdown compared to system swap.

(\ref{parallel}) The \textsl{@parallel} decoration applied to the \textsl{calc\_norm00} function invokes the parallel interface, which means that the single set of $(N, M)$ function arguments will be replaced by a list of multiple $(N1, M1), (N2, M2),..., (Nk, Mk)$ arguments, see (\ref{defcalcnorm00}) and (\ref{arglist}). The parallel interface will then execute, simultaneously, multiple copies of \textsl{calc\_norm00} running on multiple ($ncpus=8$) CPU cores in parallel. The argument list is automatically divided up and distributed among the parallel processes, and in the end, the results are reassembled in a single list.

(\ref{inmost}) This is the innermost loop---the rate limiting computation of integer factorization. The  \textsl{factor} function (\ref{factor}) is a wrapper around the standard \textsl{PARI} factorization routine \cite{PARIlibrary}. Note that we did not use specialized factorization algorithms suitable for certain classes of numbers, because the computation of  $L_\infty (N)$ requires the factorization of every single number $1-N$. The number of iterations in the innermost loop should be hand tuned to achive optimum balance between computation and data manipulation (see more about this below). Also note that determining whether or not a number $N$ is prime can, of course, be calculated much faster than full factorization but \textsl{factor} has already been called on $N$.

(\ref{input_beg}-\ref{input_end})  The input parameters $N\_*$ allow flexibility to compute $L_\infty (N)$ depending on the hardware configuration. $N\_min$ and $N\_max$ define a particular segment, in our computations we used 20 segments each $5*10^{10}$ long to get to $10^{12}$. The workflow is organized in three nested loops. We already mentioned the innermost loop (\ref{inmost}), which includes multiple factorizations calculated in a single call to \textsl{calc\_norm00}. The inner loop (\ref{inner}) iterates multiple calls to \textsl{calc\_norm00} using the parallel interface as explained above (\ref{parallel}). Note that the argument list (\ref{arglist}) is a standard \textsl{Python} list, which has a significant memory footprint and, therefore, the balance between the chosen values of $N\_inmost\_loop$ and $N\_inner\_loop$ is crucial. i) Their product $N\_maxmem$ should be set such that the resulting \textsl{Python} list (\ref{arglist}) fits comfortably in memory. The associated work arrays (\ref{array_beg}-\ref{array_end}) are all \textsl{bignumpy} arrays and their size is not limiting. ii) $N\_inmost\_loop$ should be set $>>1$ to make sure that computation (factorization) dominates because the overhead of data manipulation associated with a single function call to \textsl{calc\_norm00} in the parallel environment is significant. In fact, the maximum overall speedup using 8 CPU cores was limited to less than five fold.

(\ref{L00_array}) The core data set at the heart of our computations is comprised of the $L_\infty (N)$ values at prime ``stops'' along the number trail. The associated \textsl{bignumpy} array can be generated piece wise by concatenating consecuitve sub-arrays generated in a succession of batch calculations as noted above. Since the file associated with a \textsl{bignumpy} array is the exact binary copy of the array's memory image, these files can readily be concatenated using the UNIX \textsl{cat} command.

(\ref{hopseq}, \ref{cumsum}, \ref{prn_hopseq}, \ref{prn_cumsum})  $L_\infty (N)$ is computed as the cumulative sum of the sequence of ``hops'' between consecutive numbers along the number trail and, therefore, every new batch computation needs two data points from the previous calculation to start with. One is the length of the last hop and the other is the current value of $L_\infty (N)$.

(\ref{outer_beg}-\ref{outer_end}) The cumulative summation is carried out in the outer loop utilizing a number of \textsl{numpy} operations. i) First, the sequence of hops is computed for a continuous segment $N_{beg}-N_{end}$ by taking the maximum of every two adjacent infinity norm values (Chebyshev contour indices), and the values are stored in \textsl{Hop\_Sequence\_Arr} (\ref{np_max}). ii) The cumulative sum is then computed in two steps (\ref{cumsum2}, \ref{+=cumsum}). (The current values of \textsl{last\_hop\_seq} and \textsl{cumsum} are saved for the next segment (\ref{hopseq2}, \ref{hopseq3}, \ref{cumsum3}), see previous paragraph.) iii) Finally, \textsl{Prime\_bIndex\_Arr} is utilized as a binary mask (\ref{mask}) to keep only the prime stops and store them in \textsl{Prime\_Stops\_onL00\_Arr} (\ref{prime_stops}).

(\ref{arglist}) Note that the output list generated from the return values of the parallel calculation (\ref{ret_lst}) is not guaranteed to preserve the order of the input list, and must be sorted.

(\ref{beg_cumsum}-\ref{end_cumsum}) Listing of the \textsl{last\_hop\_seq} and \textsl{cumsum} values printed at lines \ref{prn_hopseq}, \ref{prn_cumsum} after completion of each of the 20 segments of the master calculation.

(\ref{beg_diffs}-\ref{end_diffs}) This section of the code reads in the concatenated file \textsl{Prime\_Stops\_onL00\_\_N=1-1000000000000.mmap} holding the $L_\infty(p)$ values at the prime ``stops'' and computes the $\mathcal D^1$ and $\mathcal D^2$ differences.

(\ref{histplot}) The final section computes and plots the $\mathcal D^1$ and $\mathcal D^2$ histograms.

\lstset{
  language=Python,
  showstringspaces=false,
  keywordstyle=\color{black}\bfseries,
  commentstyle=\itshape,
  basicstyle={\ttfamily,\small},
  stringstyle=\slshape,
  numbers=left, numberstyle=\tiny,
  emph={find, calc_norm00, factor,prime_pi},
  emphstyle=\underline,
  breaklines,
  breakatwhitespace,
  breakindent=20pt,
  escapeinside={(*@}{@*)}
}
\begin{lstlisting}
import numpy as np
from bignumpy import bignumpy(*@\label{bignumpy}@*)
####################################################################

def find(name, path):
    for root, dirs, files in os.walk(path):
        if name in files:
            return os.path.join(root, name)
####################################################################

# Calc for N: Chebyshev contour index (norm00)
@parallel(p_iter='multiprocessing', ncpus=8)(*@\label{parallel}@*)
def calc_norm00(N,M):(*@\label{defcalcnorm00}@*)

    N += 1                # N is passed as a zero based array index
    norm00     = []
    is_a_prime = []
    if N==1:
        norm00     = [0]
        is_a_prime = [0]

    for i in xrange( max( 2,N),N+M):               # innermost loop(*@\label{inmost}@*)

        f = factor( i)(*@\label{factor}@*)
        tuples = [x for x   in f]
       #primes = [p for p,e in f]
        powers = [e for p,e in f]
        norm00.append( max( powers))               # infinity norm
        if len( tuples)==1 and tuples[0][1]==1:    # 'i' is a prime
            is_a_prime.append( 1)
        else:
            is_a_prime.append( 0)

    return norm00, is_a_prime(*@\label{ret_lst}@*)
####################################################################

%timeit
N_min         = 450000000001(*@\label{input_beg}@*)
N_max         = 500000000000
N_primes      = prime_pi( N_max) - prime_pi( N_min-1)
N_outer_loop  = 5*10^3
N_inner_loop  = 10^4
N_inmost_loop = 10^3
N_maxmem      = N_inner_loop * N_inmost_loop          # max array size stored in memory(*@\label{input_end}@*)
ext           = str(N_min)+"-"+str( N_max)+".mmap"    # file extension

# Chebyshev contour index array --create(*@\label{array_beg}@*)
Contour_Indx_mmap_fname = "Contour_Indx__N="+ext
Contour_Indx_Arr = bignumpy( Contour_Indx_mmap_fname, np.uint8, (N_maxmem,))

# Is it a prime? Binary index array --create
Prime_bIndex_mmap_fname = "Prime_bIndex__N="+ext
Prime_bIndex_Arr = bignumpy( Prime_bIndex_mmap_fname, np.uint8,  N_maxmem,))

# Hop sequence array --create
Hop_Sequence_mmap_fname = "Hop_Sequence__N="+ext
Hop_Sequence_Arr = bignumpy( Hop_Sequence_mmap_fname, np.uint8, (N_maxmem,))

# Temp array for data type casting --create(*@\label{array_end}@*)
Temp1_Arr = bignumpy( 'Temp1.mmap', np.int64, (N_maxmem,))
Temp2_Arr = bignumpy( 'Temp2.mmap', np.int64, (N_maxmem,))

# L00 values at prime stops --create(*@\label{L00_array}@*)
Prime_Stops_onL00_mmap_fname = "Prime_Stops_onL00__N="+ext
Prime_Stops_onL00_Arr = bignumpy( Prime_Stops_onL00_mmap_fname, np.int64, (N_primes,))

incr = N_inner_loop * N_inmost_loop
stp  = N_inmost_loop
last_prime_stop = 0
last_hop_seq    = 11(*@\label{hopseq}@*)
cumsum          = 1029766280643(*@\label{cumsum}@*)
tim  = walltime()
for i in xrange( N_outer_loop):                           # OUTER LOOP(*@\label{outer_beg}@*)

    beg        = i *   incr + N_min-1     # array indx
    end        = beg + incr
    input_list = [(x,stp) for x in range(beg,end,stp)]    # inner loop(*@\label{inner}@*)
    tuples = sorted( list( calc_norm00( input_list)))(*@\label{arglist}@*)

    output_list = [ x[1][0] for x in tuples ]
    flat_list = [item for sublist in output_list for item in sublist]
    Contour_Indx_Arr = np.fromiter( flat_list, np.uint8)
    if beg==0:
        Hop_Sequence_Arr[0] = 0
    else:
        Hop_Sequence_Arr[0] = max( Contour_Indx_Arr[0], last_hop_seq)(*@\label{hopseq2}@*)
    last_hop_seq = Contour_Indx_Arr[-1](*@\label{hopseq3}@*)
    np.maximum( Contour_Indx_Arr[1:], Contour_Indx_Arr[:-1], out=Hop_Sequence_Arr[1:])(*@\label{np_max}@*)

    output_list = [ x[1][1] for x in tuples ]
    flat_list   = [item for sublist in output_list for item in sublist]
    Prime_bIndex_Arr = np.fromiter( flat_list, np.uint8)

    Temp1_Arr[:] = Hop_Sequence_Arr
    Temp1_Arr.cumsum( out=Temp2_Arr)(*@\label{cumsum2}@*)
    Temp2_Arr += cumsum(*@\label{+=cumsum}@*)
    cumsum = Temp2_Arr[-1](*@\label{cumsum3}@*)
    np.multiply( Prime_bIndex_Arr, Temp2_Arr, out=Temp1_Arr)  # binary mask(*@\label{mask}@*)
    count_nonzero = np.count_nonzero( Temp1_Arr)
    Prime_Stops_onL00_Arr[last_prime_stop:last_prime_stop+count_nonzero] = Temp1_Arr[np.nonzero( Temp1_Arr)](*@\label{prime_stops}@*)
    last_prime_stop += count_nonzero(*@\label{outer_end}@*)

    if( (i+1) % (N_outer_loop // 100)) == 0:
        print str(int(100*float(i+1)/float(N_outer_loop))).rjust(3)+'%', 'Wall time = ', str(walltime(tim)).rjust(20)


   #print Contour_Indx_Arr[...]
   #print Prime_bIndex_Arr[...]
   #print Hop_Sequence_Arr[...]

#print Prime_Stops_onL00_Arr[...]

print "last_hop_seq= ", last_hop_seq(*@\label{prn_hopseq}@*)
print "cumsum=       ", cumsum(*@\label{prn_cumsum}@*)
####################################################################

#                             last_hop_seq    cumsum(*@\label{beg_cumsum}@*)
#            1-  50000000000  11              114418475903
#  50000000001- 100000000000  11              228836951528
# 100000000001- 150000000000  11              343255427228
# 150000000001- 200000000000  12              457673901868
# 200000000001- 250000000000  12              572092378372
# 250000000001- 300000000000  11              686510853733
# 300000000001- 350000000000  11              800929329310
# 350000000001- 400000000000  13              915347805373
# 400000000001- 450000000000  11             1029766280643
# 450000000001- 500000000000  12             1144184756086
# 500000000001- 550000000000  11             1258603231525
# 550000000001- 600000000000  12             1373021707238
# 600000000001- 650000000000  11             1487440183820
# 650000000001- 700000000000  11             1601858659108
# 700000000001- 750000000000  12             1716277134939
# 750000000001- 800000000000  14             1830695610669
# 800000000001- 850000000000  11             1945114085159
# 850000000001- 900000000000  11             2059532560893
# 900000000001- 950000000000  11             2173951036235
# 950000000001-1000000000000  12             2288369511216(*@\label{end_cumsum}@*)
####################################################################


N_max = 10^12
N_primes = prime_pi( N_max)(*@\label{beg_diffs}@*)
ext = str( N_max)+".mmap"
dir = "/where/to/find/my_files/"

Prime_Stops_onL00_mmap_fname = "Prime_Stops_onL00__N=1-"+ext
Prime_Stops_onL00_Arr = bignumpy( find( Prime_Stops_onL00_mmap_fname, dir), np.int64, (N_primes,))        # L00 at prime stops --read

Prime_diff1_onL00_mmap_fname = "Prime_diff1_onL00__N=1-"+ext
Prime_diff1_onL00_Arr = bignumpy( Prime_diff1_onL00_mmap_fname, np.int16, ( N_primes-1,))                 # --create
np.subtract( Prime_Stops_onL00_Arr[1:], Prime_Stops_onL00_Arr[:-1], out=Prime_diff1_onL00_Arr)

Prime_diff2_onL00_mmap_fname = "Prime_diff2_onL00__N=1-"+ext
Prime_diff2_onL00_Arr = bignumpy( Prime_diff2_onL00_mmap_fname, np.int16, ( N_primes-2,))                 # --create
np.subtract( Prime_diff1_onL00_Arr[1:], Prime_diff1_onL00_Arr[:-1], out=Prime_diff2_onL00_Arr)(*@\label{end_diffs}@*)
####################################################################

# Plot data

import matplotlib.pyplot as plt(*@\label{histplot}@*)

E_max    = 12
N_max    = 10^12
N_primes = prime_pi( N_max)

ext = str(N_max)+".mmap"
dir = "/where/to/find/my_files/"

Prime_diff1_onL00_mmap_fname = "Prime_diff1_onL00__N=1-"+ext
Prime_diff1_onL00_Arr = bignumpy( find( Prime_diff1_onL00_mmap_fname, dir), np.int16, ( N_primes-1,))                # --read

Prime_diff2_onL00_mmap_fname = "Prime_diff2_onL00__N=1-"+ext
Prime_diff2_onL00_Arr = bignumpy( find( Prime_diff2_onL00_mmap_fname, dir), np.int16, ( N_primes-2,))                # --read

slice_diff1 = [prime_pi( 10^x)-1 for x in range( E_max+1)]
slice_diff2 = [prime_pi( 10^x)-2 for x in range( E_max+1)]

bin_edges_Prime_diff1_onL00 = np.arange( np.amin( Prime_diff1_onL00_Arr),np.amax( Prime_diff1_onL00_Arr)+2)            # '+2' to make sure rightmost bin is empty
bin_edges_Prime_diff2_onL00 = np.arange( np.amin( Prime_diff2_onL00_Arr),np.amax( Prime_diff2_onL00_Arr)+2)

# Histogram raw data:

slice_indx = 12  # >= 1
sdiff1 = slice_diff1[slice_indx]
sdiff2 = slice_diff2[slice_indx]

hist, bin_edges = np.histogram( Prime_diff1_onL00_Arr[:sdiff1], bins=bin_edges_Prime_diff1_onL00)
start = np.amin( np.nonzero(hist))
end   = np.amax( np.nonzero(hist))
np.savetxt( 'Prime_diff1_onL00_histogram__N=1-'+str( 10^slice_indx)+'.txt', hist[start:end+1],      fmt='%12d')
np.savetxt( 'Prime_diff1_onL00_bin_edges__N=1-'+str( 10^slice_indx)+'.txt', bin_edges[start:end+1], fmt='%12d')

hist, bin_edges = np.histogram( Prime_diff2_onL00_Arr[:sdiff2], bins=bin_edges_Prime_diff2_onL00)
start = np.amin( np.nonzero( hist))
end   = np.amax( np.nonzero( hist))
np.savetxt( 'Prime_diff2_onL00_histogram__N=1-'+str(10^slice_indx)+'.txt', hist[start:end+1],      fmt='%12d')
np.savetxt( 'Prime_diff2_onL00_bin_edges__N=1-'+str(10^slice_indx)+'.txt', bin_edges[start:end+1], fmt='%12d')

# Individual plots:

slice_indx = 12  # >= 1
my_color = (188/255, 27/255, 60/255)

plt.figure(1)
plt.hist( Prime_diff1_onL00_Arr[:slice_diff1[slice_indx]], bins=bin_edges_Prime_diff1_onL00, color=my_color, normed=True, align='left', histtype='stepfilled')
plt.ylabel( 'Frequency', fontsize=20)
plt.yticks( fontsize=16)
plt.xticks( fontsize=16)

plt.figure(2)
plt.hist( Prime_diff2_onL00_Arr[:slice_diff2[slice_indx]], bins=bin_edges_Prime_diff2_onL00, color=my_color, normed=True, align='left', histtype='stepfilled')
plt.ylabel( 'Frequency', fontsize=20)
plt.yticks( fontsize=16)
plt.xticks( fontsize=16)

# Multiple plots:

my_colors = [(0/255, 153/255, 153/255), (252/255, 126/255, 0/255), (251/255, 2/255, 1/255), (102/255, 205/255, 204/255), (255/255, 175/255, 103/255), (255/255, 102/255, 102/255), (0/255, 103/255, 102/255), (178/255, 86/255, 1/255), (175/255, 1/255, 2/255), (83/255, 165/255, 161/255), (246/255, 113/255, 78/255)]

plt.figure(3)
plt.hist( (Prime_diff1_onL00_Arr[:slice_diff1[2]], Prime_diff1_onL00_Arr[:slice_diff1[3]], Prime_diff1_onL00_Arr[:slice_diff1[4]], Prime_diff1_onL00_Arr[:slice_diff1[5]], Prime_diff1_onL00_Arr[:slice_diff1[6]], Prime_diff1_onL00_Arr[:slice_diff1[7]], Prime_diff1_onL00_Arr[:slice_diff1[8]], Prime_diff1_onL00_Arr[:slice_diff1[9]], Prime_diff1_onL00_Arr[:slice_diff1[10]], Prime_diff1_onL00_Arr[:slice_diff1[11]], Prime_diff1_onL00_Arr), bins=bin_edges_Prime_diff1_onL00, color=my_colors, normed=False, align='left', log=True, stacked=True, histtype='stepfilled')
plt.ylabel( 'log Frequency', fontsize=20)
plt.yticks( fontsize=16)
plt.xticks( fontsize=16)

plt.figure(4)
plt.hist( (Prime_diff2_onL00_Arr[:slice_diff2[2]], Prime_diff2_onL00_Arr[:slice_diff2[3]], Prime_diff2_onL00_Arr[:slice_diff2[4]], Prime_diff2_onL00_Arr[:slice_diff2[5]], Prime_diff2_onL00_Arr[:slice_diff2[6]], Prime_diff2_onL00_Arr[:slice_diff2[7]], Prime_diff2_onL00_Arr[:slice_diff2[8]], Prime_diff2_onL00_Arr[:slice_diff2[9]], Prime_diff2_onL00_Arr[:slice_diff2[10]], Prime_diff2_onL00_Arr[:slice_diff2[11]], Prime_diff2_onL00_Arr), bins=bin_edges_Prime_diff2_onL00, color=my_colors, normed=False, align='left', log=True, stacked=True, histtype='stepfilled')
plt.ylabel( 'log Frequency', fontsize=20)
plt.yticks( fontsize=16)
plt.xticks( fontsize=16)

plt.show()
####################################################################
\end{lstlisting}

\end{document}